\documentclass[12pt]{amsart}
\usepackage[toc,page]{appendix}
\usepackage{amssymb, latexsym, euscript}
\usepackage{amssymb}
\usepackage{latexsym}
\usepackage{amsmath}

\def\sD{{\mathfrak D}}      
   \def\sH{{\mathfrak H}}   
   \def\sK{{\mathfrak K}}   \def\sL{{\mathfrak L}}
\def\sM{{\mathfrak M}}   \def\sN{{\mathfrak N}}   
      \def\sR{{\mathfrak R}}

\def\st{{\mathfrak t}}

      \def\dC{{\mathbb C}}
\def\dD{{\mathbb D}}

   \def\dN{{\mathbb N}}   
      \def\dR{{\mathbb R}}

   \def\cB{{\mathcal B}}   \def\cC{{\mathcal C}}
\def\cD{{\mathcal D}}      
   \def\cH{{\mathcal H}}   
   \def\cK{{\mathcal K}}   \def\cL{{\mathcal L}}
\def\cM{{\mathcal M}}   \def\cN{{\mathcal N}}   
   \def\cQ{{\mathcal Q}}   \def\cR{{\mathcal R}}
\def\cS{{\mathcal S}}      
\def\cV{{\mathcal V}}      \def\cX{{\mathcal X}}
\def\cY{{\mathcal Y}}   \def\cZ{{\mathcal Z}}

   \def\bB{{\mathbf B}}

\def\cRS{\mathcal{RS}}

\def\half{{\frac{1}{2}}}

\def\wt{\widetilde}

\def\d1{{\mathcal D}}
\def\CM{{\cM}}

\def\half{{\frac{1}{2}}}

\def\RE{{\rm Re\,}}
\def\IM{{\rm Im\,}}

\def\ran{{\rm ran\,}}
\def\cran{{\rm \overline{ran}\,}}
\def\dom{{\rm dom\,}}
\def\mul{{\rm mul\,}}

\def\cdom{{\rm \overline{dom}\,}}

\def\cspan{{\rm \overline{span}\, }}

\def\cmr{{\dC \backslash \dR}}
\def\uphar{{\upharpoonright\,}}

\def\wt{\widetilde}

\def\f{\varphi}

\def\half{{\frac{1}{2}}}

\newtheorem{theorem}{Theorem}[section]
\newtheorem{lemma}[theorem]{Lemma}
\newtheorem{proposition}[theorem]{Proposition}
\newtheorem{corollary}[theorem]{Corollary}
\newtheorem{definition}[theorem]{Definition}
\newtheorem{remark}[theorem]{Remark}
\newtheorem{example}[theorem]{Example}
\numberwithin{equation}{section}

\def\RE{{\rm Re\,}}
\def\IM{{\rm Im\,}}

\def\wt{\widetilde}
\def\wh{\widehat}

\def\f{\varphi}

\def\uphar{{\upharpoonright\,}}

\begin{document}

\title[Stieltjes and inverse Stieltjes families]
{Representations of closed quadratic forms associated with Stieltjes and inverse Stieltjes holomorphic families of linear relations}

\author[Yury Arlinski\u{\i}]{Yu.M.~Arlinski\u{\i}}
\address{Stuttgart, Germany}

\email{yury.arlinskii@gmail.com}
\author[Seppo Hassi]{S. Hassi}
\address{Department of Mathematics and Statistics  \\
University of Vaasa  \\
P.O. Box 700  \\
65101 Vaasa  \\
Finland} \email{sha@uwasa.fi}


\vskip 1truecm
\thispagestyle{empty}
\baselineskip=12pt

\dedicatory{Dedicated to our friend and colleague Vladimir Derkach on the occasion of his
seventieth birthday}



\begin{abstract}
In this paper holomorphic families of linear relations which belong to the Stieltjes or inverse Stieltjes class are studied.
It is shown that in their domain of holomorphy $\dC\setminus\dR_+$ the values of Stieltjes and inverse Stieltjes families are, up to a rotation, maximal sectorial. This leads to a study of the associated closed sesquilinear forms and their representations.
In particular, it is shown that the Stieltjes and inverse Stieltjes holomorphic families of linear relations are of type (B) in the sense of Kato. These results are proved by using linear fractional transforms which connect these families to holomorphic functions that belong to a combined Nevanlinna-Schur class and a key tool then relies on a specific structure of contractive operators.
\end{abstract}

\maketitle

\section{Introduction}

The main objective of the present paper is a further study of \textit{Stieltjes} and  \textit{inverse Stieltjes} operator-valued functions, or more generally, Stieltjes and inverse Stieltjes holomorphic families of linear relations (l.r. for short), see \cite{ArlHassi_2019[2]}. Elements in these classes are Nevanlinna functions \cite{AllenNarcow1976,ArlBelTsek2011, BHS2020, Br, DM2017, GesTsek2000, KacK, Shmul71} or, more generally, Nevanlinna families \cite{DHMS06, DdeS, KL1,KL2,LT} which admit holomorphic continuation to the negative semi-axis $\dR_-$.
Recall, cf. e.g. \cite{BHS2020,DHMS06}, that a family of l.r.'s $\CM(\lambda)$, $\lambda \in \cmr$, in a Hilbert space $\sM$ is called a \textit{Nevanlinna family} if:

(1) $\CM(\lambda)$ is maximal dissipative for every $\lambda \in \dC_+$
(resp. accumulative for every $\lambda \in \dC_-$);
(2) $\CM(\lambda)^*=\cM(\bar \lambda)$, $\lambda \in \cmr$;
(3) for some, and hence for all, $\mu \in \dC_+ (\dC_-)$ the
      operator family $(\CM(\lambda)+\mu)^{-1} (\in [\sM])$ is
      holomorphic for all $\lambda \in \dC_+ (\dC_-)$.

The class of all Nevanlinna families in a Hilbert space $\sM$ is denoted by $\wt R(\sM)$.
Each Nevanlinna familiy $\CM$ admits the decomposition
\begin{equation}\label{hfpkj}
  \CM(\lambda)=M_s(\lambda)\oplus M_\infty,\quad
  M_\infty=\{0\}\times\mul \CM(\lambda).
\end{equation}
to the constant multi-valued part $M_\infty$ and the operator part $M_s(\lambda)$ (a Nevanlinna family of densely defined
operators in $\sM \ominus \mul \CM(\lambda)$,
$\lambda\in\cmr$).

Stieltjes and inverse Stieltjes families are particular type of Nevanlinna families being defined as follows.

\begin{definition}\label{invStieltjes}\cite{ArlHassi_2019[2]}.
A family of l.r.'s $\CM(\lambda)$ in a Hilbert space $\sM$ defined on an open connected set $\lambda \in \dC\setminus\dR_+$ is said to be a \textit{Stieltjes family} (respectively, \textit{inverse Stieltjes family})
if it is a Nevanlinna family for $\lambda\in\cmr$ and, moreover,
\begin{enumerate}
\item for all $x<0$ the l.r.'s $\cM(x)$ are selfadjoint and nonnegative (respectively, nonpositive),
\item the family $\cM(\lambda)$ is holomorphic on $\dR_-$, i.e., for any $x<0$ and for some $\xi\in\rho(\cM(x))$ (and hence for all $\xi\in\rho(\cM(x)))$ the resolvent $\left(\cM(\lambda)-\xi I\right)^{-1}$ exists and is holomorphic in $\lambda$ from a neighborhood of $x$, depending on $\xi$.
\end{enumerate}
\end{definition}

In what follows the classes of all Stieltjes and inverse Stieltjes families
in a Hilbert space $\sM$ are denoted by $\wt \cS(\sM)$ and $\wt\cS^{(-1)}(\sM)$, respectively.

The following statements are equivalent, cf. \cite{KacK,ArlHassi_2019[2]},
\begin{enumerate}
\def\labelenumi{\rm (\roman{enumi})}
\item $\cM(\lambda)$ is Stieltjes/inverse Stieltjes family,
\item $-\cM(1/\lambda)$ is inverse Stieltjes/respect. Stieltjes family,
\item $-\cM(\lambda)^{-1}$ is inverse Stieltjes/respect.  Stieltjes family,
\item $\lambda\cM(\lambda)$ $/\cfrac{1}{\lambda}\cM(\lambda)$ is inverse Stieltjes/respect. Stieltjes family.
\end{enumerate}

In our paper \cite{ArlHassi_2019[2]} Stieltjes and inverse Stieltjes families were connected to operator-valued holomorphic functions defined on the domain $\dC\setminus \{(-\infty,-1]\cup [1,\infty)\}$. More precisely it was shown that there is a one-to-one correspondences
given by a linear fractional transformation of functions and their variables between the classes of Stieltjes/inverse Stieltjes families in the Hilbert space $\sM$ and operator-valued functions from the combined Nevanlinna-Schur class $\cRS(\sM)$ (being recently studied in \cite{ArlHassi_2018}); see Lemma \ref{TH1} below.

Let us briefly describe the main results in the present paper.

(1) We prove that an arbitrary Stieltjes/inverse Stieltjes family  $\cM(\lambda)$ is a holomorphic family of the type (B) \cite{Ka} (Theorem \ref{typB}), i.e., $\cM(\lambda)$ is maximal sectorial l.r. with vertex at the origin and an acute semi-angle in each point $\lambda\in\dC\setminus\dR_+$, the domain $\cD[\cM(\lambda)]=\cD$ of the associated closed form does not depend on $\lambda\in\dC\setminus\dR_+$, and the quadratic form $\cM(\lambda)[u]$ is holomorphic for $\lambda\in\dC\setminus\dR_+$ for every $u\in \cD$.

(2) Let $\sM$ and $\sK$ be a Hilbert spaces and let $\cZ$ be a closed not necessarily densely defined linear operator in $\sM$.
Moreover, let $\wh A$ be a nonnegative selfadjoint l.r. in $\sK$ and let $V\in\bB(\sM,\sK)$ be a contraction.
Then the holomorphic closed sesquilinear sectorial form
\[\begin{array}{l}
q(\lambda)[u,v]=\left((I_\sM+(1+\lambda)V^*(\wh A-\lambda I)^{-1}V)\cZ u,\cZ v\right)_\sM,\; \\u,v\in\dom\cZ,
\; \lambda\in\dC\setminus\dR_+
\end{array}
\]
defines by the first representation theorem \cite{Ka} the Stieltjes family in $\sM$ of the form
\begin{equation}\label{Q1}
\cQ(\lambda)=\cZ^*\left(I_\sM+(1+\lambda)V^*(\wh A-\lambda I)^{-1}V\right)\cZ,
\end{equation}
while the holomorphic closed sesquilinear sectorial form
\[
\begin{array}{l}
r(\lambda)[u,v]:=-q(\lambda^{-1})=
\left(\left(-I_\sM-\left(1+\frac{1}{\lambda}\right)V^*\left(\wh A-\frac{1}{\lambda} I\right)^{-1}V\right)\cZ u,\cZ v\right)_\sM,\\
u,v\in \dom\cZ,\; \lambda\in\dC\setminus\dR_+
\end{array}
\]
defines by the first representation theorem the inverse Stieltjes family in $\sM$ of the form
\begin{equation}\label{R1}
 \cR(\lambda)=\cZ^*\left(-I_\sM-\left(1+\frac{1}{\lambda}\right)V^*\left(\wh A-\frac{1}{\lambda} I\right)^{-1}V\right)\cZ.
\end{equation}
Here $\cZ^*$ is the adjoint l.r., which is an operator if and only if $\cZ$ is densely defined.

The converse statements are also true: if $\cQ$ is a Stieltjes ($\cR$ is an inverse
Stieltjes  family) in $\sM$, then there exist an auxiliary Hilbert space
$\sK$, a closed not necessarily densely defined linear operator
$\cZ$ in $\sM$, a nonnegative selfadjoint l.r. $\wh A$ in $\sK$, and
a contraction $V\in\bB(\sM,\sK)$ such that the associated closed
form $\cQ(\lambda)[\cdot,\cdot]$ ($\cR(\lambda)[\cdot,\cdot]$) takes the form
$q(\lambda)[\cdot,\cdot]$ ($r(\lambda)[\cdot,\cdot]$), see Proposition
\ref{stutcor1}, Theorem \ref{typB}.

(3) From the above representations we derive in Theorem \ref{intrep1}
the integral representations of Stieltjes and inverse Stieltjes $\bB(\sM)$-valued functions
similar to the scalar case \cite{KacK}.

(4) The existence of strong resolvent limits at $-0$ and $-\infty$ is proved for any Stieltjes/inverse Stieltjes family
and their properties are established (Proposition \ref{jbfzo}).

(5) Examples of the Stieltjes/inverse Stieltjes families $Q(\lambda)$ possessing the property
\[
\dom Q(\lambda)\cap\dom Q(\mu)=\{0\},\\
\forall \lambda,\mu\in\dC\setminus\dR_+,\;\lambda\ne\mu
\]
are constructed. 

We systematically use the properties of l.r.'s established
in \cite{Ar, codd1974, CS, DdeS, RoBe}, relationships given by the
linear fractional transformations between selfadjoint contractions
and nonnegative l.r.'s and their resolvents, transfer
functions of selfadjoint passive system and Stieltjes/inverse
Stieltjes families.

Let us mention that in \cite{ArlBelTsek2011, ArlHassi_2015, AHS2, ArlKlotz2010, KalWo, KL1, KL2, KL, LT}
the results related to the representations of certain classes of  operator-valued Nevanlinna functions/Nevanlinna families
as compressed resolvents of selfadjoint exit space extensions were obtained.
Realizations of some subclasses of Stieltjes and inverse Stieltjes matrix-valued functions as the impedance functions of singular $L$-systems have been considered in \cite{ArlBelTsek2011,BeTse2, Tsek1994}.

 On the other hand, Stieltjes and inverse Stieltjes functions whose values are unbounded operators often appear in the analysis of partial differential operators, for instance, when considering Laplace operators on bounded (or unbounded) domains $\Omega\subset \dR^n$ with a smooth or non-smooth boundary.
A well-known object in their study is the ($\lambda$-dependent) Dirichlet-to-Neumann map $D(\lambda)$. After a sign change one gets the map
$-D(\lambda)$, which in the case of a smooth boundary is actually an inverse Stieltjes function on the boundary space $L^2(\partial\Omega)$,
whose values are unbounded operators with a constant domain not depending on $\lambda$. The inverse function $D(\lambda)^{-1}$ is a Stieltjes function,
whose values are compact operators on $L^2(\partial\Omega)$; cf. \cite{BeLa07} and \cite[Proposition~7.6]{DHM17}. When considering
in this setting the so-called Kre\u{\i}n Laplacian (associated with the Kre\u{\i}n-von Neumann extension of the underlying minimal operator),
which involves so-called regularized trace mappings, one obtains a Stieltjes function, whose values are unbounded operators.
Moreover, it is shown in \cite[Proposition~7.15]{DHM17}, \cite[Theorem~3.12]{DHM20b} that when the boundary gets very irregular
the inverse function $D(\lambda)^{-1}$ can become even multivalued in which case $D(\lambda)^{-1}$ is a domain invariant Stieltjes family.

{\bf Notations.}
Throughout this paper separable Hilbert spaces over the field $\dC$
of complex numbers are considered. The symbols $\dom T$, $\ran T$, $\ker T$, $\mul T$
stand for the domain, the range, the null-subspace, and the multivalued part of
a linear relation $T$. The closures of $\dom T$, $\ran T$ are denoted by $\cdom T$,
$\cran T$, respectively. The identity operator in a Hilbert space
$\sH$ is denoted by  $I$ and sometimes by $I_\sH$. If $\sL$ is a
subspace, i.e., a closed linear subset of $\sH$, the orthogonal
projection in $\sH$ onto $\sL$ is denoted by $P_\sL.$ The notation
$T\uphar \cN$ means the restriction of a linear operator $T$ to the
set $\cN\subset\dom T$. The resolvent set of a l.r. $T$ is denoted by
$\rho(T)$ and ${\rm s-R}-\lim$ means the strong resolvent limit of
linear relations \cite[Chapter 8, \S 1]{Ka}, \cite[Chapter 1]{BHS2020}.
The linear space of bounded operators acting between the Hilbert spaces
$\sH$ and $\sK$ is denoted by $\bB(\sH,\sK)$ and the Banach algebra $\bB(\sH,\sH)$ by
$\bB(\sH).$ For a contraction
$T\in\bB(\sH,\sK)$ the defect operator $(I-T^*T)^{\half}$ is denoted
by $D_T$ and $\sD_T:=\cran D_T$. For defect operators one has the
commutation relations from \cite{SF}: $TD_T = D_{T^*}T, \quad T^*D_{T^*}=D_{T}T^*.$

\section{Preliminaries}
\subsection{Sectorial linear relations}
Let $A=\left\{\{f,f'\}\right\}$ be a linear relation (l.r.)
in the Hilbert space $\sH$ with the inner product $(\cdot,\cdot).$ Then
\[
W(A)=\{(f',f):\; f\in\dom A,\;||f||=1\}
\]
is called  numerical range of $A$.
A l.r. $A$ is called ({\bf a}) accretive if $W(A)$ is contained in the closed right half-plane $\RE \lambda\ge 0$,
({\bf b}) $m$-accretive if it is accretive and has no accretive extensions in $\sH$; see
\cite{Ka,RoBe} (equivalently: the resolvent set $\rho(A)$ contains the left half-plane/ the adjoint l.r. $A^*$ is accretive).
Moreover, $A$ is said to be \textit{sectorial with the vertex at the origin and the semi-angle $\alpha\in[0,\pi/2)$}
if
\[
W(A)\subseteq\left\{z\in\dC:\;|\arg z|\le\alpha\right\}=:S(\alpha).
\]
Clearly, a sectorial l.r. is accretive; it is called \textit{maximal sectorial}, or \textit{$m$-$\alpha$-sectorial} for short, if
it is $m$-accretive.

Each $m-accretive$ (respectively, $m-\alpha$-sectorial l.r.) $A$ can be represented as follows \cite{RoBe}
$
A={\rm Gr} (A_s) \oplus \{0, \mul A\},
$ 
where $A_s$ is the operator part of $A$ acting on the subspace $\sH\ominus \mul A$ and ${\rm Gr} (A_s)=\{\{f, A_s f\},\; f\in\dom A_s=\dom A\}.$ Observe, in particular, that the case $\alpha=0$ corresponds to the class of nonnegative l.r. $A\ge 0$.

Let $-\pi<\alpha<\beta\le\pi$, $\beta-\alpha<\pi$. If $A$ is a l.r.,
\[
W(A)\subseteq\{z\in\dC:\;\alpha\le\arg z\le \beta\},
\]
and $A$ has regular points outside the above sector, then $A$  also will said to be $m$-sectorial.
Note that
\[
W(\exp{-i(\beta+\alpha)} A)\subseteq\left\{\dC:\;|\arg z|\le(\beta-\alpha)/2\right\}=S((\beta-\alpha)/2),
\]
i.e., the l.r. $A_{\alpha,\beta}:=\exp{-i(\beta+\alpha)}A$ is $m-(\beta-\alpha)/2$-sectorial.
A similar definition can be given for sectorial sesquilinear forms.  A sectorial form ${\mathfrak a}$ is called closed, see \cite{Ka} if
\begin{multline*}
\lim\limits_{n\to\infty}f_n=f\;\mbox{and}\lim\limits_{m,n\to\infty}{\mathfrak a}[f_n-f_m]=0\; \mbox{for}\;\{f_n\}\subset\dom{\mathfrak a}\\
\Longrightarrow f\in\dom{\mathfrak a},\;
\lim\limits_{n\to\infty}{\mathfrak a}[f-f_n]=0.
\end{multline*}
There is a one-to-one correspondence between closed sectorial forms $\tau$ with vertex at the origin and  $m$-sectorial l.r.'s ${A}$ (the first representation theorem \cite{Ka}):
\begin{multline*}
\dom A=\left\{u\in\dom\tau: \exists \f\quad\mbox{such that}\quad {\mathfrak a}[u,v]=(\f,v)\;\forall v\in\dom{\mathfrak a}\right\}, \\
A=\{u,\f+(\sH\ominus\cdom{\mathfrak a})\}.
\end{multline*}
In what follows the closed sesquilinear form associated with an $m$-sectorial l.r. $A$ is
denoted by $A[\cdot,\cdot]$ and its domain by $\cD[A]$. The form $A[\cdot,\cdot]$ is the closure of the form \cite{RoBe}
\[
{\mathfrak a}[f,g]:=(f',g)=(A_s f,g),\; \{f, f'\},\;\{g, g'\}\in A.
\]

\subsection{Linear fractional transformations of sectorial linear relations}
\begin{definition} \cite{Arl1987}.
Let $\alpha\in (0,\pi/2)$ and let $T$ be a linear operator in the Hilbert space $H$
defined on $H$. If
\begin{equation}
\label{CA1} ||T\sin\alpha\pm i\cos\alpha I_H||\le 1,
\end{equation}
then we say that $T$ belongs to the class
$C_H(\alpha)$.
\end{definition}
The condition \eqref{CA1} is equivalent to
\begin{equation}
\label{CA2} 2|\IM (Tf,f)|\le\tan\alpha(||f||^2-||Tf||^2),\;
f\in\dom T.
\end{equation}
Therefore, if $T\in C_H(\alpha)$, then $T$ is a contraction. Due to
\eqref{CA2} it is natural to consider selfadjoint
contractions in $H$ as operators of the class
$C_H(0)$. In view of \eqref{CA2} one can write that $C_H(0)=\underset{\alpha\in(0,\pi/2)}{\bigcap}C_H(\alpha).$

Analogously, set
$C(\alpha):=\left\{z\in\dC:|z\sin\alpha\pm i\cos\alpha|<1\right\}$.
If $\alpha=0$, then $C(0)=[-1,1].$

For a l.r. $A$ consider the following linear fractional transformation (the Cayley transform):
\begin{equation}\label{CALTR}
 \cC(A):=-I+2 (A+I)^{-1}=\left\{\{f+f',f-f'\}:
\{f,f'\}\in A\right\}.
\end{equation}
Clearly,  $\cC(\cC(A))=A$ and, moreover, $\dom \cC(A)=\ran(I+A)$ and $\ran \cC(A)=\ran(I-A)$.
The transformation \eqref{CALTR} establishes a one-to-one correspondence between
$m-\alpha$-sectorial l.r.'s in $H$ and the class
$C_H(\alpha)$.
In addition, $T\in C_H(\alpha)$ if and only if the bounded operator
$(I-T^*)(I+T)$ is a
sectorial operator with the vertex at the origin and the semiangle
$\alpha$; see \cite{Arl1991}.

Properties of operators of the class $\wt C_H:=\underset{\alpha\in[0,\pi/2)}{\bigcup}C_H(\alpha)$ were studied in
\cite{Arl1987, Arl1991}.

Let $A$ be an $m-\alpha$-sectorial l.r. in $\sH$. Then
the closed sectorial form $A[u,v]$ associated with $A$ can be described as follows.
Let $T=\cC(A)$. Then $T\in\wt C_\sH$ and hence the operators $I_\sH\pm T$ are
m-sectorial (bounded) operators. It follows that
\[
 I_\sH+T=(I_\sH+T_R)^{\half}(I+iG)(I_\sH+T_R)^{\half},
\]
where $T_R=(T+T^*)/2$ is the real part of $T$, $G$ is a bounded
selfadjoint operator in the subspace $\cran(I_\sH+T_R)^{\half}$, and $I$
is the identity operator in $\cran(I_\sH+T_R)^{\half}$.
\begin{proposition}\cite[Proposition 9.5.1]{ArlBelTsek2011}, \cite[Proposition 2.2]{ArlHassi_2015}.\label{clfrm}
The closed sectorial form associated with an m-sectorial l.r. $A$
is given by
\[
\begin{array}{l}
\cD[A]=\ran (I+ T_R)^{\half},
\\ A[u,v]=-(u,v)+2\left((I+iG)^{-1}(I_\sH+T_R)^{[-\half]}u,(I_\sH+T_R)^{[-\half]}v\right),
\end{array}
\]
where $(I_\sH+T_R)^{[-\half]}$ is the Moore--Penrose pseudo-inverse, and $u,v\in\cD[A]$.
\end{proposition}
Notice that if $T=\cC(A)=-I+2(I+A)^{-1}$, then $(A+I)^{-1}=\frac{1}{2}\left(T+I\right)$ and
the resolvents are connected by the following identities
\begin{equation}
\label{connectresab}\begin{array}{l}
 (A-\lambda I)^{-1}=-\frac{1}{1+\lambda}\left(I+\frac{2}{1+\lambda}\left(T-\frac{1-\lambda}{1+\lambda}\, I\right)^{-1}\right),\quad \lambda\in\rho(A)\setminus\{-1\},\\[3mm]
\Longleftrightarrow\;
 (A-\lambda I)^{-1}=\cfrac{1}{1-\lambda}(T+I)\left(I-\cfrac{1+\lambda}{1-\lambda}T\right)^{-1},\quad \lambda\in\rho(A)\setminus\{1\}.
\end{array}
\end{equation}

\subsection{Operator-valued functions of the class $\cRS(\sM)$}

\begin{definition}\label{rsm} \cite{ArlHassi_2018}.
Let $\sM$ be a separable Hilbert space. A $\bB(\sM)$-valued Nevanlinna function $\Omega$ holomorphic on $\dC\setminus\{(-\infty,-1]\cup[1,+\infty)\}$ is said to belong to the class $\cRS(\sM)$ (combined Nevanlinna-Schur class) if $-I\le
\Omega(x)\le I$ for $x\in (-1,1)$.
\end{definition}
If $\Omega\in\cRS(\sM)$ then according to \cite[Theorem 5.1, Proposition 5.6]{AHS_CAOT2007} (see also \cite{ArlHassi_2018})
there exists (up to unitary equivalence) a unique minimal passive selfadjoint system
$$\tau=\left\{ \begin{bmatrix} D&C\cr C^*&F\end{bmatrix},\sM,\sM,\sK\right\}$$
whose transfer function coincides with $\Omega(z)$, i.e., the following relation holds:
\[
\Omega(z)=\Omega_\tau(z)=D+z C(I-zF)^{-1}C^*,\;z\in\dC\setminus\{(-\infty,-1]\cup[1,+\infty)\}.
\]
Moreover, if
$
T=\begin{bmatrix} D&C\cr C^*&F\end{bmatrix}:
\begin{array}{l} \sM \\\oplus\\ \sK \end{array} \to
\begin{array}{l} \sM \\\oplus\\ \sK \end{array}
$ 
is the selfadjoint contraction corresponding to the system $\tau$, then
due to the Schur-Frobenius formula for the resolvent of a block operator matrix one has
\[
P_\sM(I-zT)^{-1}\uphar\sM=(I_\sM-z\Omega_\tau(z))^{-1},\;\;z\in\dC\setminus\{(-\infty,-1]\cup[1,+\infty)\}.
\]

Note that because $\tau$ is a passive discrete-time system, its transfer function $\Omega_\tau$ belongs to the Schur class \cite{A}, i.e. $||\Omega_\tau(z)||\le 1$ for all $z$ in the unit disk $\dD$.

Recall that a selfadjoint passive system $\tau=\left\{ \begin{bmatrix} D&C\cr C^*&F\end{bmatrix},\sM,\sM,\sK\right\}$ is called minimal if
\[
\cspan\left\{F^n C^*\sM,\;n\in\dN\cup\{0\}\right\}=\sM\oplus\sK.
\]
The latter is equivalent to the fact that the operator $T$ acting in $\sH=\sM\oplus\sK$ is $\sM$-minimal.
It follows from the definition that if $\Omega\in \cRS(\sM)$,
then strong limit values $\Omega(\pm 1)$ exist, see \cite{Arl1991}.
Observe, that if $\Omega\in\cRS(\sM)$, then because it is a Schur class operator-valued function
and also a Nevanlinna function on the domain $\dC\setminus\{(-\infty,-1]\cup[1,+\infty)\}$
it admits the following block operator representation for all $z\in\dC\setminus\{(-\infty,-1]\cup[1,+\infty)\}$,
\begin{equation}\label{struct11}\begin{array}{l}
\Omega(z)=\begin{bmatrix}I&0&0\cr 0&-I&0\cr 0&0&\omega_0(z)  \end{bmatrix}:\begin{array}{l} \ker(I-\Omega(0))\\\oplus\\\ker(I+\Omega(0))\\\oplus\\\sD_{\Omega(0)}\end{array}\to \begin{array}{l} \ker(I-\Omega(0))\\\oplus\\\ker(I+\Omega(0))\\\oplus\\\sD_{\Omega(0)}\end{array},
\end{array}
\end{equation}
where $\omega_0$ belongs to the class $\cRS(\sD_{\Omega(0)})$ and  $\ker D_{\omega_0(0)}=\{0\}.$

The class $\cRS(\sM)$ can be characterized also as follows.
\begin{theorem}
\label{newchar} \cite[Theorem~4.1]{ArlHassi_2018}. Let $\Omega$ be an operator-valued
Nevanlinna function defined on the domain $\dC\setminus\{(-\infty,-1]\cup[1,+\infty)\}$.
Then the following statements are equivalent:
\begin{enumerate}
\def\labelenumi{\rm (\roman{enumi})}
\item $\Omega$
 belongs to the class $\cRS(\sM)$;
\item $\Omega$ satisfies the inequality
\begin{equation}
\label{ythfd}
 I-\Omega^*(z)\Omega(z)-(1-|z|^2)\cfrac{\IM \Omega (z)}{\IM z}\ge
 0,\quad \IM z\ne 0;
\end{equation}

\item the kernel
\begin{equation}\label{ythfd22}
K(z,w):=I-\Omega^*(w)\Omega(z)-\cfrac{1-\bar w z}{z-\bar w}\,(\Omega(z)-\Omega^*(w)
\end{equation}
is nonnegative on the domains
$$\dC\setminus\{(-\infty,-1]\cup [1,\infty)\},\; \IM z>0\quad\mbox{and}\quad  \dC\setminus\{(-\infty,-1]\cup [1,\infty)\},\; \IM z<0.$$
\end{enumerate}
\end{theorem}

Observe that the operator $\Omega(z)$ belongs to the class $\wt C_\sM$ when $z\in\dD$. More precisely from \eqref{ythfd} we get
\[
 2\left|\IM\left(\Omega(z)h,h\right)\right|\le 2\cfrac{|\IM z|}{1-|z|^2}\left(||h||^2-\left\|\Omega(z)h\right\|^2\right)\;\forall h\in\sM,\;\forall z\in\dD.
\]
This means that
\begin{equation}\label{omca}
\Omega(z)\in C_\sM(\alpha_z)\;\mbox{where}\quad
\alpha_z=\arctan\left(\cfrac{2|\IM z|}{1-|z|^2}\right),\; z\in\dD.
\end{equation}

Let $\sH,\,\sK,$ $\sM$, and $\sN$ be Hilbert spaces.
The following well-known result gives a parametrization of all contractive block operators
acting from $\sH\oplus \sM$ into $\sK\oplus\sN.$

\begin{proposition} \cite{AG, DaKaWe, ShYa}. \label{ParContr}
The operator matrix
$
T=\begin{bmatrix} D&C \cr B&F\end{bmatrix} :
\begin{array}{l} \sM \\\oplus\\ \sK \end{array} \to
\begin{array}{l} \sN \\\oplus\\ \cL \end{array}
$
is a contraction if and only if $D\in\bB(\sM,\sN)$ is a contraction
and the entries $B$, $C$, $F$ take the form
\[
 B=ND_D,\quad C=D_{D^*}G,\quad F=-ND^*G+D_{N^*}LD_{G},
\]
where the operators $N\in\bB(\sD_D,\sK)$, $G\in\bB(\cL,\sD_{D^*})$
and $L\in\bB(\sD_{G},\sD_{N^*})$ are contractions. Moreover, the
operators $N,\,G,$ and $L$ are uniquely determined by $T$.
\end{proposition}
\begin{remark}
\label{stut1} If $\sN=\sM$ and $\cL=\sK$, then $T\in\bB(\sM\oplus\sK)$
is a selfadjoint contraction if and only if $D=D^*$, $B=C^*$,
$G=N^*$, $L=L^*$.
\end{remark}

All the statements in the next lemma follow from Proposition \ref{ParContr} and Remark \ref{stut1}.

\begin{lemma}\label{LEM1}
(1) Let
$
T=\begin{bmatrix} D&C\cr
C^*&F\end{bmatrix}:\begin{array}{l}\sM\\\oplus\\
\sK\end{array}\to
\begin{array}{l}\sM\\\oplus\\
\sK\end{array}
$
be a selfadjoint contraction with entries parameterized as follows
\begin{equation}
\label{parame}
 C^*= ND_D,\; C= D_D N^*,\;
 F=-NDN^*+D_{N^*}XD_{N^*},
\end{equation}
where $N:\sD_D\to\sK$ is a contraction and $X$ is a selfadjoint
contraction in $\sD_{N^*}$. Define
\begin{equation}\label{ff}
\begin{array}{c}
 F':=F+N(I+D)N^*=NN^*+D_{N^*}XD_{N^*},\\
 F'':=F-N(I-D)N^*=-NN^*+D_{N^*}XD_{N^*}.
\end{array}
\end{equation}
Then $F', F''$ are selfadjoint contractions in $\sK$ such that
$
F'-F''=2NN^*.
$

(2) Let
\[
 B(z)=F+zC^*(I-zD)^{-1}C,\quad z\in\dC\setminus\{(-\infty,-1]\cup [1,+\infty)\}.
\]
Then the following identities hold:
\[
F'=B(1):=s-\lim\limits_{x\uparrow1}B(x), \quad
F''=B(-1):=\lim\limits_{x\downarrow -1}B(x).
\]
(3) Define
\begin{equation}
\label{zzz}
\begin{array}{l}
\Sigma_+(z):=z(I+D)^{\half}N^*(I-zF)^{-1}NP_{\sD_D}(I+D)^{\half},\\
\Sigma_-(z):=z(I-D)^{\half}N^*(I-zF)^{-1}NP_{\sD_D}(I-D)^{\half},\\
z\in\dC\setminus\{(-\infty,-1]\cup[1,+\infty)\}.
\end{array}
\end{equation}
Then
\[
\begin{array}{l}
\left(I_\sM-\Sigma_+(z)\right)^{-1}
=I_{\sM}+z(I+D)^{\half}N^*(I-zF')^{-1}NP_{\sD_D}(I+D)^{\half},\\
\left(I_{\sM}+\Sigma_-(z)\right)^{-1}
=I_{\sM}-z(I-D)^{\half}N^*(I-zF'')^{-1}NP_{\sD_D}(I-D)^{\half},\\
z\in\dC\setminus\{(-\infty,-1]\cup[1,+\infty)\}.
\end{array}
\]
(4) Moreover, if
$
 W(z)=I+zDN^*\left(I-z\cfrac{F'+F''}{2}\right)^{-1}N,
$ 
then
\[
W(z)^{-1}=I-zDN^*(I-zF)^{-1}N,\quad z\in\dC\setminus\{(-\infty,-1]\cup[1,+\infty)\}.
\]
\end{lemma}

\subsection{The class $\cRS(\sM)$ and Stieltjes and inverse Stieltjes families}
The next statement is established in \cite[Lemma 3.2]{ArlHassi_2019[2]}.
It gives relationships between holomorphic operator-valued functions from the combined Nevanlinna-Schur class $\cRS(\sM)$,
see Definition \ref{rsm}, and Stieltjes and inverse Stieltjes families of l.r.'s in the Hilbert space $\sM$.

\begin{lemma} \cite{ArlHassi_2019[2]} \label{TH1}
Let $\Omega\in\cRS(\sM)$. Then for all $\lambda\in \dC\setminus\dR_+$,
\begin{equation}
\label{formula1}
\begin{array}{rl}
\cQ(\lambda)&=-I+2\left(I_\sM-\Omega\left(\cfrac{1+\lambda}{1-\lambda}\right)\right)^{-1}\\
            &=\left\{\left\{\left(I_\sM-\Omega\left(\cfrac{1+\lambda}{1-\lambda}\right)\right)h,
    \left(I_\sM+\Omega\left(\cfrac{1+\lambda}{1-\lambda}\right)\right)h\right\}:\; h\in\sM\right\}
\end{array}
\end{equation}
is a Stieltjes family and
\begin{equation}
\label{formula2}
\begin{array}{rl}
\cR(\lambda)&=I-2\left(I_\sM+\Omega\left(\cfrac{1+\lambda}{1-\lambda}\right)\right)^{-1}\\
            &=\left\{\left\{\left(I_\sM+\Omega\left(\cfrac{1+\lambda}{1-\lambda}\right)\right)h,
    \left(\Omega\left(\cfrac{1+\lambda}{1-\lambda}\right)-I_\sM\right)h\right\}:\; h\in\sM\right\}
\end{array}
\end{equation}
is an inverse Stieltjes family.

Conversely, if $\cQ(\lambda)$ is a Stieltjes family (resp., $\cR(\lambda)$ is an inverse Stieltjes family) in $\sM$,
then there exists a function $\Omega\in\cRS(\sM)$ such that
\eqref{formula1} (resp., \eqref{formula2}) holds.

Furthermore, the functions $\cQ$ in \eqref{formula1} and $\cR$ in \eqref{formula2} are connected by $\cR=-\cQ^{-1}$
and thus $\cQ\in\wt\cS(\sM)$ if and only if $-\cQ^{-1}\in\wt\cS^{(-1)}(\sM)$.
\end{lemma}

\section{Sectorialness of Stieltjes and inverse Stieltjes families}

Let $\cQ(\lambda)$ be a Stieltjes family in $\sM$.
Then by Lemma \ref{TH1} there is $\Omega(z)\in\cRS(\sM)$ such that
\[
\cQ(\lambda)=-I+2\left(I-\Omega\left(\cfrac{1+\lambda}{1-\lambda}\right)\right)^{-1}.
\]
Clearly $\mul\cQ(\lambda)=\ker(I-\Omega(0))=\sM\ominus\cdom \cQ(-1)$. Moreover,
using \eqref{hfpkj} and \eqref{struct11}, we see that
$
\cQ(\lambda)=\cQ_s(\lambda)\oplus \cQ_\infty,
$
where
\[
\cQ_\infty=\{0\}\times\ker(I-\Omega(0)),
\]
and $\cQ_s(\lambda)$ is an operator-valued function from the Stieltjes class $\wt \cS(\cdom\cQ(-1))$.
If $\{f,f'\}\in \cQ(\lambda)$ then for the quadratic form $(f',f)$ we use the expressions
\[
 \cQ(\lambda)[f]:=(f',f)=(\cQ_s(\lambda)f,f).
\]
Similarly for an inverse Stieltjes family $\cR(\lambda)$ on $\sM$ the representation \eqref{formula2} in Lemma \ref{TH1}
and \eqref{struct11} show that $\mul \cR(\lambda)=\ker(I+\Omega(0))=\sM\ominus\cdom \cR(-1)$ and
$
\cR(\lambda)=\cR_s(\lambda)\oplus \cR_\infty,
$
where
\[
\cR_\infty=\{0\}\times\ker(I+\Omega(0)),
\]
and $\cR_s(\lambda)$ is an operator-valued function from the inverse Stieltjes class $\wt \cS^{(-1)}(\cdom\cR(-1))$.
Moreover, for the quadratic form one has
\[
 \cR(\lambda)[f]:=(f',f)=(\cR_s(\lambda)f,f), \quad \{f,f'\}\in \cR(\lambda).
\]
Recall that $S(\varphi)$ stands for the sector
$$ S(\varphi):=\{\lambda\in \dC:|\arg \lambda|\le\varphi\},\;\varphi\in[0,\pi/2).$$

\begin{theorem}\label{cbvgfn}
Let $\cQ(\lambda)$ be a Nevanlinna family in the Hilbert space $\sM$ which is holomorphic on $\dC\setminus\dR_+$.
Then the following assertions are equivalent:
\begin{enumerate}
\def\labelenumi{\rm (\roman{enumi})}
\item $\cQ(\lambda)$ is a Stieltjes family;
\item the following inequality holds
\[
\RE\cQ(\lambda)[f]+\cfrac{\RE\lambda}{|\IM\lambda|}\,\,|\IM\cQ(\lambda)[f]|\ge 0,\;f\in\dom\cQ(\lambda),\; \IM\lambda\ne 0;
\]
\item $\lambda\cQ(\lambda)$ is a Nevanlinna family in $\sM$.
\end{enumerate}

Furthermore, a Stieltjes family $\cQ(\lambda)$, $\lambda\in\dC\setminus\dR_+$, in the Hilbert space $\sM$ admits the following properties:

If $\arg\lambda\in[-\pi,\pi)$, then
\[
\begin{array}{l}
\RE\lambda<0\Rightarrow
|\IM\cQ(\lambda)[f]|\le\frac{|\IM\lambda|}{-\RE\lambda}\RE\cQ(\lambda)[f]\;\forall f\in\dom\cQ(\lambda)\\[3mm]
\qquad\qquad \Longleftrightarrow W(\cQ(\lambda))\subset S(\pi-|\arg\lambda|);\\[3mm]
\RE\lambda=0\Rightarrow \RE \cQ(\lambda)[f]\geq 0\;\forall f\in\dom\cQ(\lambda)
\Longleftrightarrow \cQ(\lambda)\quad\mbox{is an m-accretive l.r.};\\[3mm]
\RE\lambda>0,\; \left\{ \begin{array}{l}
    \IM\lambda>0\Rightarrow W\left(\exp(-i(\pi-|\arg\lambda|)/2)\cQ(\lambda)\right)\subset S({\frac{\pi-|\arg\lambda|}{2}}),\\[3mm]
    \IM\lambda<0\Rightarrow W\left(\exp(i(\pi-|\arg\lambda|)/2)\cQ(\lambda)\right)\subset S({\frac{\pi-|\arg\lambda|}{2}}).
    \end{array}\right.
\end{array}
\]
In particular, if $\beta\in[0,\pi/2)$, then
\[
\begin{array}{l}
\arg\lambda\in [\pi-\beta,\pi+\beta]\Longrightarrow W(\cQ(\lambda))\subset S(\beta),\\
\arg\lambda\in [\beta,\pi/2]\;(\beta\neq 0)\,\Longrightarrow W(\cQ(\lambda))\subset\{\xi\in\dC:\arg\xi\in [0,\pi-\beta]\},\\
\arg\lambda\in [-\pi/2,-\beta]\;(\beta\neq 0)\,\Longrightarrow W(\cQ(\lambda))\subset\{\xi\in\dC:\arg\xi\in [\pi+\beta,2\pi]\}.
\end{array}
\]
\end{theorem}

\begin{proof}
First the equivalence of assertions (i) -- (iii) is shown.

(i) $\Longleftrightarrow$ (ii) Denote $z=\frac{1+\lambda}{1-\lambda}$, $\lambda\in\dC\setminus \dR_+$.
Then equivalently $z\in \dC\setminus \left((-\infty,-1] \cup [1,\infty)\right)$ and one has
\begin{equation}\label{zlambda}
 \frac{1-|z|^2}{\IM z}=-\frac{2 \RE \lambda}{\IM \lambda}, \quad \IM \lambda \neq 0.
\end{equation}
By Lemma \ref{TH1} $\cQ(\lambda)$ is given by the transform \eqref{formula1} for some $\Omega\in\cRS(\sM)$,
in fact,
\[
 \Omega(z)=\left\{\left\{\left(I_\sM+\cQ\left(\lambda\right)\right)h,
    \left(\Omega\left(\lambda\right)-I_\sM\right)h\right\}:\; h\in\sM\right\},
    \quad \lambda=\cfrac{z-1}{z+1}.
\]
Therefore, $\{f,f'\}\in \cQ(\lambda)$ precisely when $\{g,g'\}:=\{f+f',f'-f\}\in \Omega(z)$.
Since $\|g\|^2+\|g'\|^2=4 \RE(f',f)$ and $\IM (g',g)=2\IM (f',f)$ we conclude from
Theorem \ref{newchar}  that the inequality \eqref{ythfd} with $\{g,g'\}\in \Omega(z)$ is equivalent to
\[
 \RE (f',f) + \frac{\RE \lambda}{\IM \lambda}\, \IM (f',f) \geq 0.
\]
Since $\cQ(\lambda)$ is a Nevanlinna family, one has $\dfrac{\IM (f',f)}{\IM \lambda}\geq 0$.
This proves the claim.

(ii) $\Longleftrightarrow$ (iii) Let $\{f,f'\}\in \cQ(\lambda)$.
Then equivalently $\{f,\lambda f'\}\in \lambda\cQ(\lambda)$ and
\[
 \IM (\lambda f',f) = \RE\lambda\; \IM (f',f) + \IM\lambda\; \RE (f',f).
\]
Therefore the inequality in (ii) means that $\CM(\lambda):=\lambda \cQ(\lambda)$
is dissipative for $\lambda\in\dC_+$ and accumulative for $\lambda\in\dC_-$.
Since $\cQ(\lambda)$ is a Nevanlinna family, one has $\CM(\lambda)^*=\CM(\bar\lambda)$, $\lambda\in\cmr$.
Moreover, $\cQ(\lambda)$ is maximal dissipative for $\lambda\in\dC_+$, i.e.,
$\ran (\cQ(\lambda)+\mu)=\sM$ for some, equivalently for all,
$\mu \in\dC_+$ and $(\cQ(\lambda)+\mu)^{-1}$ is holomorphic as a function of $\lambda\in\dC_+$.
Then also $\ran (\CM(\lambda)+\lambda\mu)=\lambda\, \ran (\cQ(\lambda)+\mu)=\sM$ and
$(\CM(\lambda)+\lambda\mu)^{-1}=\frac{1}{\lambda} (\cQ(\lambda)+\mu)^{-1}$ is holomorphic at $\lambda$,
when $\mu$ is chosen such that $0<\arg \mu <\pi-\arg \lambda$. One concludes that $\CM(\lambda)=\lambda \cQ(\lambda)$
is a Nevanlinna family if and only if $\cQ(\lambda)$ satisfies the inequality in (ii).
Hence, the equivalence of (i) -- (iii) is shown.

Now the remaining assertions can be easily proved. The statements with $\RE \lambda<0$ and $\RE \lambda=0$ are clear from
the inequality in (ii).

Assume that $\RE \lambda>0$. Since $\cQ(\lambda)$ and $\lambda\cQ(\lambda)$ are Nevanlinna families
one has $\IM (\cQ(\lambda)f,f)\geq 0$ and $\IM (\lambda\,\cQ(\lambda)[f])\geq 0$ for all $f\in\dom\cQ(\lambda)$ when $\IM \lambda >0$.
In this case $0\leq \arg \cQ(\lambda)[f] + \arg\lambda \leq \pi$ or, equivalently,
\[
 -\frac{\pi-\arg\lambda}{2} \leq \arg\left(\exp(-i(\pi-\arg\lambda)/2)\cQ(\lambda)[f]\right) \leq \frac{\pi-\arg\lambda}{2},
\]
where $0<\arg\lambda<\pi/2$. Similarly, if $\IM \lambda<0$ then $-\pi\leq \arg \cQ(\lambda)[f] + \arg\lambda \leq 0$ or, equivalently,
\[
 -\frac{\pi-|\arg\lambda|}{2} \leq \arg\left(\exp(i(\pi-|\arg\lambda|)/2)\cQ(\lambda)[f]\right) \leq \frac{\pi-|\arg\lambda|}{2},
\]
where $-\pi/2<\arg\lambda<0$. This gives the assertions for $\RE \lambda>0$.

The last three implications concerning the numerical range of $\cQ(\lambda)$ are clear.
\end{proof}

\begin{theorem}
\label{cbvgfn2}
Let $\cR(\lambda)$ be a Nevanlinna family in the Hilbert space $\sM$ which is holomorphic on $\dC\setminus\dR_+$.
Then the following assertions are equivalent:
\begin{enumerate}
\def\labelenumi{\rm (\roman{enumi})}
\item $\cR(\lambda)$ is an inverse Stieltjes family;
\item the following inequality holds
\[
\begin{array}{l}
\RE\cR(\lambda)[f]-\cfrac{\RE\lambda}{\IM\lambda}\,\,\IM\cR(\lambda)[f]\le 0\\
\Longleftrightarrow\RE(-\cR(\lambda)[f])+\cfrac{\RE\lambda}{|\IM\lambda|}\,\,|\IM(-\cR(\lambda)[f])|
\ge 0,\;f\in\dom\cR(\lambda),\; \IM\lambda\ne 0;
\end{array}
\]
\item $\dfrac{\cR(\lambda)}{\lambda}$ is a Nevanlinna family in $\sM$.
\end{enumerate}

Furthermore, an inverse Stieltjes family $\cR(\lambda)$, $\lambda\in\dC\setminus\dR_+$, in the Hilbert space $\sM$ admits the following properties:

If $\arg\lambda\in[-\pi,\pi)$, then
\[
\begin{array}{l}
\RE\lambda<0\Rightarrow
|\IM(-\cR(\lambda)[f])|\le\frac{|\IM\lambda|}{-\RE\lambda}\RE(-\cR(\lambda)[f])\;\forall f\in\dom\cR(\lambda)\\[3mm]
\iff W(-\cR(\lambda))\subset S({\pi-|\arg\lambda|});\\[3mm]
\RE\lambda=0\Longrightarrow \RE (-\cR(\lambda)[f])\ge 0\;\forall f\in\dom\cR(\lambda)
\Longleftrightarrow -\cR(\lambda)\quad\mbox{is an m-accretive l.r.};\\[3mm]
\RE\lambda>0,\; \left\{ \begin{array}{l}
     \IM\lambda>0 \Longrightarrow W\left(\exp(-i(\pi+|\arg\lambda|)/2)\cR(\lambda)\right)\subset S({\frac{\pi-|\arg\lambda|}{2}}),\\[3mm]
     \IM\lambda<0 \Longrightarrow W\left(\exp(i(\pi+|\arg\lambda|)/2)\cR(\lambda)\right)\subset S({\frac{\pi-|\arg\lambda|}{2}}).
    \end{array}\right.

\end{array}
\]
In particular, if $\beta\in[0,\pi/2)$, then
\[
\begin{array}{l}
\arg\lambda\in [\pi-\beta,\pi+\beta]\Longrightarrow W(\cR(\lambda))\subset\{\xi\in\dC:\arg\lambda\in[\pi-\beta,\pi+\beta]\},\\
\arg\lambda\in [\beta,\pi/2]\; (\beta\neq 0)\,\Longrightarrow W(\cR(\lambda))\subset\{\xi\in\dC:\arg\xi\in [\beta,\pi]\},\\
\arg\lambda\in [-\pi/2,-\beta]\;(\beta\neq 0)\,\Longrightarrow W(\cR(\lambda))\subset\{\xi\in\dC:\arg\xi\in [-\pi,-\beta]\}.
\end{array}
\]
\end{theorem}

\begin{proof}
Since the proof is similar to the proof of Theorem \ref{cbvgfn} only the main steps are pointed out here.

Using formula \eqref{formula2} in Lemma \ref{TH1} it is seen that
$\{f,f'\}\in -\cR(\lambda)$ precisely when $\{g,g'\}:=\{f+f',f-f'\}\in \Omega(z)$.
Since $\|g\|^2+\|g'\|^2=4 \RE(f',f)$ and $\IM (g',g)=-2\IM (f',f)$ we conclude from
Theorem \ref{newchar} and the formula \ref{zlambda} that the inequality \eqref{ythfd} with $\{g,g'\}\in \Omega(z)$ is equivalent to
\[
 \RE (f',f) - \frac{\RE \lambda}{\IM \lambda}\, \IM (f',f) \geq 0,
\]
i.e.,
\[
 \RE(-\cR(\lambda)[f])-\cfrac{\RE\lambda}{\IM\lambda}\,\,\IM(-\cR(\lambda)[f]) \ge 0,\quad f\in\dom\cR(\lambda),\; \IM\lambda\ne 0.
\]
This yields the equivalence (i) $\Longleftrightarrow$ (ii).

On the other hand, with $\{f,f'\}\in \cR(\lambda)$ one has
$\{f,1/\lambda f'\}\in \frac{\cR(\lambda)}{\lambda}$ and
\[
 \IM (1/\lambda f',f) = \IM \left(\overline{\lambda} (f',f)\right) = \RE\lambda\; \IM (f',f) - \IM\lambda\; \RE (f',f),
\]
or, equivalently,
\[
 \IM \left(\frac{\cR(\lambda)[f]}{\lambda}\right) =  \RE\lambda\; \IM \cR(\lambda)[f] - \IM\lambda\; \RE \cR(\lambda)[f],\quad
  f\in\dom\cR(\lambda),\; \IM\lambda\ne 0.
\]
Therefore the inequality in (ii) means that $\frac{\cR(\lambda)}{\lambda}$
is dissipative for $\lambda\in\dC_+$ and accumulative for $\lambda\in\dC_-$,
which then yields the equivalence  (ii) $\Longleftrightarrow$ (iii).

As to the remaining assertions, observe that for $\IM \lambda >0$ one has $0\leq \arg \cR(\lambda)[f]\leq \pi$ and
\[
 0\leq \arg\left(\frac{\cR(\lambda)[f]}{\lambda}\right)=\arg \cR(\lambda)[f] - \arg\lambda \leq \pi,
\]
for all $f\in \dom \cR(\lambda)$ or, equivalently,
\begin{equation}\label{ImR1}
 -\pi+\arg\lambda \leq \arg(-\cR(\lambda)[f]) \leq 0.
\end{equation}
Similarly $\IM \lambda <0$ one has $-\pi\leq \arg \cR(\lambda)[f]\leq 0$,
$-\pi \leq \arg \cR(\lambda)[f] - \arg\lambda \leq 0$, and
\begin{equation}\label{ImR2}
 0 \leq \arg(-\cR(\lambda)[f]) \leq \pi-|\arg\lambda|.
\end{equation}
All the statements in the second part of the theorem follow from the inequalities in (ii), \eqref{ImR1}, and \eqref{ImR2}.
\end{proof}

\begin{remark}
For Stieltjes and inverse Stieltjes families another proof of the equivalence (i) $\Longleftrightarrow$ (iii) in Theorems~\ref{cbvgfn} and~\ref{cbvgfn2} is given in \cite{ArlHassi_2019[2]}.
For scalar Stieltjes and inverse Stieltjes functions the corresponding equivalences can be found in \cite{KacK}.
\end{remark}

Let $\cQ$ be a Stieltjes family and let $\cR$ be an inverse Stieltjes family. Theorem \ref{cbvgfn} and Theorem \ref{cbvgfn2} show that
\begin{enumerate}
\item for any $\lambda=r\exp(i\psi)$, $\psi=\pi\pm\alpha,$ where $\alpha\in [0,\pi/2),$  the families $\cQ(\lambda)$ and $-\cR(\lambda)$ are maximal sectorial l.r.'s with the vertex at the origin and the semi-angle $\alpha$;
\item if $\lambda=r\exp(i\psi)$, where $\psi\in(0, \pi/2],$ then
\[
\begin{array}{l}
 \arg(\cQ(\lambda)[f])\in [0,\pi-\psi]\;\forall f\in\dom\cQ(\lambda)\setminus\{0\},\\
  \arg(-\cR(\lambda)[g])\in [\psi-\pi,0]\;\forall g \in\dom\cR(\lambda)\setminus\{0\};
\end{array}
\]
\item if $\lambda=r\exp(i\psi)$, where $\psi\in [-\pi/2,0)$, then
\[
\begin{array}{l}
 \arg(\cQ(\lambda)[f])\in [-\psi-\pi,0]\;\forall f\in\dom\cQ(\lambda)\setminus\{0\},\\
  \arg(-\cR(\lambda)[g])\in [0,\pi+\psi]\;\forall g \in\dom\cR(\lambda)\setminus\{0\}.
\end{array}
\]
\end{enumerate}

Thus, for an arbitrary $\lambda\in \dC\setminus\dR_+$ there exists a real number $\varphi_\lambda$ such that $\exp(i\varphi_\lambda)\cQ(\lambda)$ (respect. $\exp(i\varphi_\lambda)\cR(\lambda)$) is a maximal sectorial l.r. with the vertex at the origin and an acute semi-angle.
Therefore, every Stieljes family $\cQ(\lambda)$ and inverse Stieltjes family $\cR(\lambda)$ on $\sM$
determines a unique (in general nondensely defined) closed sectorial form (up to a rotation) via the closure of
\[
 \cQ(\lambda)[f,g]=(f',g)=(\cQ_s(\lambda)f,g), \quad \{f,f'\},\{g,g'\}\in \cQ(\lambda),
\]
and
\[
\cR(\lambda)[f,g]=(f',g)=(\cR_s(\lambda)f,g), \quad \{f,f'\},\{g,g'\}\in \cR(\lambda),
\]
respectively. By the first representation theorem this leads a one-to-one correspondence
between the closed forms $\cQ(\lambda)[f,g]$ and the representing l.r.'s $\cQ(\lambda)\in \wt\cS(\sM)$,
$\lambda\in\dC\setminus\dR_+$. Similarly there is a one-to-one correspondence
between the closed forms $\cR(\lambda)[f,g]$ and the representing l.r.'s $\cR(\lambda)\in \wt\cS^{(-1)}(\sM)$.
In Section \ref{sesqu} we prove that Stieltjes and inverse Stieltjes holomorphic families of l.r.'s
form holomorphic families of the type (B) in sense of Kato \cite{Ka}.

\section{Constructions of Stieltjes and inverse Stieltjes families}
In this section we give some explicit constructions of Stieltjes and inverse Stieltjes families.
In the first subsection the case of operator functions whose values are bounded operators on a Hilbert space $\sM$ is considered.
These results are used in second subsection to construct Stieltjes and inverse Stieltjes functions whose values are unbounded operators or, more generally, l.r.'s in $\sM$. In the last subsection special behaviour of these functions in the unbounded case is pointed out
by means of an example. The constructions given in this section are shown to be of general nature in Section \ref{sesqu}.

\subsection{Stieltjes/inverse Stieltjes functions with values in $\bB(\sM)$}

\begin{proposition}\label{totyjd} Let $\sM$ and $\sR$ be Hilbert spaces.
Assume that selfadjoint contractions $F',$ $F''$ in $\sK$, and the operator
$N\in\bB(\sM,\sK)$ are connected by the relation
\begin{equation}
\label{stutt3} F'-F''=2NN^*.
\end{equation}
Then the following identities hold for all $z\in\dC\setminus\{(-\infty,-1]\cup[1,+\infty)\}$:
\begin{equation}\label{stuttg121}
-\left(I_\sM+2\,zN^*\left(I_\sK-zF'\right)^{-1}N\right)^{-1}=-I_\sM+2zN^*\left(I_\sK-zF''\right)^{-1}N,
\end{equation}
\begin{equation}\label{stuttg1211}
\left(I_\sM-\,zN^*\left(I_\sK-z\cfrac{F'+F''}{2}\right)^{-1}N\right)^{-1}=I_\sM+zN^*\left(I_\sK-zF'\right)^{-1}N.
\end{equation}
Moreover, the operator-valued function
\begin{equation}\label{jvtufj}
\Omega_0(z):=zN^*\left(I_\sK-z\cfrac{F'+F''}{2}\right)^{-1}N,\quad
  z\in\dC\setminus\{(-\infty,-1]\cup[1,+\infty)\},
\end{equation}
belongs to the class $\cRS(\sM)$ and for all $ z\in\dC\setminus\{(-\infty,-1]\cup[1,+\infty)\}$ one has
\begin{equation}\label{lkgjvj}
\begin{array}{l}
\left(I+\Omega_0(z)\right)\left(I-\Omega_0(z)\right)^{-1}=I_\sM+2\,zN^*\left(I_\sK-zF'\right)^{-1}N, \\
\left(\Omega_0(z)-I\right)\left(I+\Omega_0(z)\right)^{-1}=-I_\sM+2zN^*\left(I_\sK-zF''\right)^{-1}N.
\end{array}
\end{equation}
\end{proposition}

\begin{proof}
From \eqref{stutt3} it follows that the block operator matrix
\[
\begin{bmatrix} 0&N^* \cr
N&\cfrac{F'+F''}{2}\end{bmatrix}:\begin{array}{l}\sM\\\oplus\\\sK\end{array}\to
\begin{array}{l}\sM\\\oplus\\\sK\end{array}
\]
is a selfadjoint contraction. Therefore, the discrete-time system
\[
\tau_0=\left\{\begin{bmatrix} 0&N^* \cr
N&\cfrac{F'+F''}{2}\end{bmatrix};\sM,\sM, \sK\right\}
\]
is passive selfadjoint. Its transfer function $\Omega_0$ is of the form
\eqref{jvtufj}.
In addition $\Omega_0(0)=0$, therefore, $\|\Omega_0(z)\|\le|z|$ for
$|z|\le 1,$ $z\ne \pm 1$, and thus $\Omega_0\in \cR\cS(\sM)$.
Using \eqref{stutt3}, the identities \eqref{stuttg121} and \eqref{stuttg1211}
can be verified by straightforward calculations.
The equalities in \eqref{lkgjvj} follow from \eqref{stuttg121} and \eqref{stuttg1211}.
\end{proof}

\begin{remark}
Let $D$ be a selfadjoint contraction in $\sM$. Then due to \eqref{stutt3} the block operator matrix
\[
T=\begin{bmatrix} D&D_DN^*\cr ND_D&-NDN^*+ \cfrac{F'+F''}{2}\end{bmatrix}
\]
is a selfadjoint contraction in $\sM\oplus\sK$ and $\tau=\{T;\sM,\sM,\sK\}$ is a passive selfadjoint discrete-time system.
Using Lemma \ref{LEM1}, item 4) (with $F=-NDN^*+(F'+F'')/2$), one obtains for the corresponding transfer function $\Omega_\tau$
the so-called M\"obius representation (cf. \cite{Shmul1})
\[
\Omega_\tau(z)=\Omega(0)+D_{\Omega(0)}\Omega_0(z)\left(I+\Omega(0)\Omega_0(z)\right)^{-1}D_{\Omega(0)},\quad
 z\in\dC\setminus\{(-\infty,-1]\cup[1,\infty)\}.
\]
\end{remark}

\begin{proposition}
\label{stu10}
Let $\sM$ and $\sK$ be Hilbert spaces.

(1) Assume that $F'$ is a selfadjoint contraction in $\sK$ and that $N'\in\bB(\sM,\cK)$.
Then the $\bB(\sM)$-valued function
\begin{equation}\label{cm1}
\cM_{01}(\lambda):=I_\sM+2\,\cfrac{1+\lambda}{1-\lambda}N'^*\left(I_\sK-\cfrac{1+\lambda}{1-\lambda}F'\right)^{-1}N',\quad \lambda\in\dC\setminus\dR_+
\end{equation}
belongs to the Stieltjes  class if and only if
\begin{equation} \label{fin11}
I_\sK+F'\ge 2 N'N'^*.
\end{equation}

(2) Assume that $F''$ is a selfadjoint contraction in $\sK$ and that $N''\in\bB(\sM,\sK)$.
Then the $\bB(\sM)$-valued function
\begin{equation}\label{cm2}
\cM_{02}(\lambda):=-I_\sM+2\,\cfrac{1+\lambda}{1-\lambda}N''^*\left(I_\sK-\cfrac{1+\lambda}{1-\lambda}F''\right)^{-1}N'',\quad \lambda\in\dC\setminus\dR_+
\end{equation}
belongs to the inverse Stieltjes class if and only if
\begin{equation}\label{fin22}
I_\sK-F''\ge 2 N''N''^*.
\end{equation}

(3) Let the selfadjoint contractions $F'$ and $F''$ in $\sK$ and the operator
$N\in\bB(\sM,\sK)$ be connected by the relation
\begin{equation} \label{fin33}
F'-F''=2NN^*.
\end{equation}
Then with $N'=N''=N$ one has $\cM_{01}\in\wt\cS(\sM)$, $\cM_{02}\in\wt\cS^{-1}(\sM)$ and, moreover,
$\cM_{02}(\lambda)=-\cM_{01}(\lambda)^{-1}$ for all $\lambda\in\dC\setminus\dR+$.

In addition, for each $\lambda\in\dC\setminus \dR_+$ there are positive constants $c_{01}(\lambda),$ $c_{02}(\lambda)$
such that the following inequalities hold:
\begin{equation}
\label{stu120}
\left|\left(\cM_{01}(\lambda)f,f\right)\right|\ge c_{01}(\lambda)\|f\|^2,\quad
\left|\left(\cM_{02}(\lambda)f,f\right)\right|\ge c_{02}(\lambda)\|f\|^2\;\;\forall f\in\sM.
\end{equation}
\end{proposition}
\begin{proof}
First observe that if a selfadjoint contraction $F'$ (respect., $F''$) in $\sK$ and
a contraction $N\in\bB(\sM,\sK)$ satisfy the inequality $I_\sK+F'\ge 2 NN^*$ (respect., $I_\sK-F''\ge 2NN^*$),
then the contraction $F'':=F'-2NN^*$ (respect., $F'=F''+2NN^*$) satisfies $I_\sK-F''\ge 2NN^*$
(respect., $I_\sK+F'\ge 2 NN^*$). Therefore, it is sufficient to prove the statement (3).

Suppose that \eqref{fin33} is valid. Then by Proposition \ref{totyjd} the function $\Omega_0$ given by \eqref{jvtufj}
belongs to the class $\cRS(\sM).$ Using \eqref{lkgjvj} it is seen that
\begin{equation}\label{momeg}
 \begin{array}{l}
\cM_{01}(\lambda)=\left(I_\sM+\Omega_0\left(\frac{1+\lambda}{1-\lambda}\right)\right)
    \left(I_\sM-\Omega_0\left(\frac{1+\lambda}{1-\lambda}\right)\right)^{-1},\\
\cM_{02}(\lambda)=\left(\Omega_0\left(\frac{1+\lambda}{1-\lambda}\right)-I_\sM\right)
    \left(I_\sM+\Omega_0\left(\frac{1+\lambda}{1-\lambda}\right)\right)^{-1},\quad
 \lambda\in\dC\setminus\dR_+.
 \end{array}
\end{equation}
Since $\Omega_0\in\cRS(\sM)$, from Lemma \ref{TH1} one concludes that $\cM_{01}$ is from the Stieltjes class and
$\cM_{02}$ is from the inverse Stieltjes class.

Conversely, the operator-valued function $\cM_{01}$ is holomorphic on $\dC\setminus\dR_+$ and it is a Nevanlinna function.
Hence it is non-decreasing on $(-\infty,0)$. Using the following well-known relation for nonnegative selfadjoint operator
$\cB$ in a Hilbert space $\sH$, cf. \cite{KrO},
\[
\lim\limits_{y\uparrow 0}\,\left((\cB-yI)^{-1}g,g\right)=\left\{
\begin{array}{ll}
    \|\cB^{[-\half]}g\|^2, & g\in\ran \cB^{\half},\\
    +\infty,           & g\in \sH\setminus\ran \cB^{\half},
\end{array}\right.,
\]
it is seen that $\lim\limits_{x\uparrow -\infty}\left(\cM_{01}(x)f,f\right)$
exists for all $f\in\sM$ if and only if $\ran N\subseteq\ran (F+I_\sK)^{\half}$.
If this is the case, then
\[
\lim\limits_{x\downarrow
-\infty}\left(\cM_{01}(x)f,f\right)=||f||^2-2\|(F'+I_\sK)^{[-\half]}N'f\|^2\;\;\forall
f\in\sM.
\]
Hence,
\[
\begin{array}{ll}
    &\lim\limits_{x\downarrow -\infty}\left(\cM_{01}(x)f,f\right)\ge 0 \quad\mbox{for all}\quad f\in\sM\setminus\{0\}\\
\iff & 2\|(F'+I_\sK)^{[-\half]}N'f\|^2\le ||f||^2 \quad\mbox{for all}\quad f\in\sM\setminus\{0\}\\
\iff & F'+I_\sK\ge 2 N'N^{'*}.
\end{array}
\]

For each $\lambda\in\dC\setminus\dR_+$ one has
$\cM_{01}(\lambda),\;\cM_{01}^{-1}(\lambda)\in\bB(\sM)$. If
$\RE\lambda<0$, then by Theorem \ref{cbvgfn} the operator
$\cM_{01}(\lambda)$ is bounded sectorial and
has bounded inverse. It follows that the operators
$\RE\cM_{01}(\lambda)$, $-\RE\cM_{02}(\lambda)$ are positive
definite. Hence for all $f\in\sM$
\[
\begin{array}{c}
\left|\left(\cM_{01}(\lambda)f,f\right)\right|\ge b_1(\lambda),\;\RE
\left(\cM_{01}(\lambda)f,f\right)\ge c_{01}(\lambda)||f||^2,\\[3mm]
\left|\left(\cM_{02}(\lambda)f,f\right)\right|\ge b_2(\lambda),\;\left|\RE
\left(\cM_{02}(\lambda)f,f\right)\right|\ge c_{02}(\lambda)||f||^2.
\end{array}
\]
Thus, \eqref{stu120} is valid for $\RE\lambda<0$.

If $\RE\lambda\ge 0$, $\lambda\notin\dR_+$, then there is $\f\in \dR$ such that
$\exp(i\f)\cM_0(\lambda)$ is $m$-sectorial and $m$-accretive. Therefore, one can use the same arguments
and thus the inequalities in \eqref{stu120} hold for such $\lambda$, too.
\end{proof}

Observe that if \eqref{fin11} and \eqref{fin22} are valid, then it follows from Lemma \ref{TH1}, the representations \eqref{formula1}, \eqref{formula2}, \eqref{momeg}, and the property \eqref{ythfd22} of the functions from the class $\cRS(\sM)$ that the operator-valued functions of two variables
\[
\cK_{01}(\lambda,\mu)=\left\{\begin{array}{l}
\cM_{01}(\lambda)+\cM_{01}(\bar\mu)+\cfrac{\lambda+\bar\mu}{\lambda-\bar\mu}\Biggl(\cM_{01}(\lambda)-\cM_{01}(\bar\mu)\Biggr),
\;\lambda,\mu\in\dC\setminus\dR_+,\;
\mu\ne\bar\lambda\\[3mm]
2\cM_{01}(\lambda)+2\cfrac{d\cM_{01}(\lambda)}{d\lambda},\;\mu=\bar\lambda,\,\lambda\in\dC\setminus\dR_+
\end{array}\right.
\]
and
\[
\cK_{02}(\lambda,\mu)=\left\{\begin{array}{l}
\cfrac{\lambda+\bar\mu}{\lambda-\bar\mu}
 \Biggl(\cM_{02}(\lambda)-\cM_{02}(\bar\mu)\Biggr)-\cM_{02}(\lambda)-\cM_{02}(\bar\mu),\;\lambda,\mu\in\dC\setminus\dR_+,\,
\mu\ne\bar\lambda\\[3mm]
2\lambda\cfrac{d\cM_{02}(\lambda)}{d\lambda}-2\cM_{02}(\lambda),\;\mu=\bar\lambda,\;\lambda\in\dC\setminus\dR_+
\end{array}\right.
\]
are nonnegative kernels on the domain $\dC\setminus\dR_+$; cf. Theorems~\ref{cbvgfn},~\ref{cbvgfn2}.

Using the Cayley transforms $A'=\cC(F')$, $A''=\cC(F'')$, the expressions \eqref{cm1} and \eqref{cm2} for the Nevanlinna families $\cM_{01}$ and $\cM_{02}$ can be rewritten as follows
\[
\begin{array}{l}
\cM_{01}(\lambda)=I_{\sM} +(1+\lambda)N'^*\Biggl(I+(1+\lambda)(A'-\lambda I)^{-1}\Biggr)N',\\
\cM_{02}(\lambda)=-I_{\sM} +(1+\lambda)N''^*\Biggl(I+(1+\lambda)( A''-\lambda I)^{-1}\Biggr)N''.
\end{array}
\]
In terms of $A'$ and $A''$ inequalities \eqref{fin11} and \eqref{fin22} take the form
$(I+A')^{-1}\ge N'N'^*
$
 and
\[
I- (I+ A'')^{-1}\ge N''N''^*\Longleftrightarrow (I+ (A'')^{-1})^{-1}\ge N''N''^*,
\]
respectively.
In particular, if $N'=N''=N$ and the equality
\[
(A'+I)^{-1}-( A''+I)^{-1}=NN^*
\]
holds, then
\[
\cM_{02}(\lambda)=-\cM_{01}(\lambda)^{-1}\;\forall \lambda\in\dC\setminus\dR_+.
\]
Observe that one can easily prove that the inequality $(I+ A')^{-1}\ge N'N'^*$ is equivalent to the conditions
$$\ran N'\subset\dom (A')^{\half},\; || (A')^{\half}N'f||^2+ ||N'f||^2\le ||f||^2 \;\;\forall f\in\sK$$
while the inequality $(I+ (A'')^{-1})^{-1}\ge N''N''^*$ is equivalent to
$$\ran N''\subset\ran (A'')^{\half},\;||(A'')^{-\half}N''f||^2+ ||N''f||^2\le ||f||^2\;\; \forall f\in\sK.$$

\begin{proposition}
Let $\sM$ and $\sK$ be Hilbert spaces.
Assume that $\wh A$ is a nonnegative selfadjoint relation in $\sK$ and $V\in\bB(\sM,\sK)$ is a contraction.
Moreover, let $F=\cC(\wh A)=-I_\sK+2(I_\sK+\wh A)^{-1}$  be the Cayley transform of $\wh A$.
Then the $\bB(\sM)$-valued function
\begin{equation}\label{dfcf1}
 \begin{array}{l}
\stackrel {0}{\cQ}_{\wh A,V}(\lambda):= I_\sM+(1+\lambda)V^*(\wh A-\lambda I_\sK)^{-1}V\\[3mm]
\qquad\qquad  =  D^2_{V}+\cfrac{2}{1-\lambda}V^*\left(I_\sK-\cfrac{1+\lambda}{1-\lambda}\,F\right)^{-1}V,\;\lambda\in\dC\setminus\dR_+,
\end{array}
\end{equation}
belongs to the Stieltjes class $\wt \cS(\sM)$. Therefore, the operator-valued function
\begin{equation}\label{dfcf2}
\begin{array}{l}
 \stackrel{0}{\cR}_{\wh A, V}(\lambda):= -\stackrel {0}{\cQ}_{\wh A,V}(1/\lambda)=-I_\sM-(1+\lambda)V^*(\lambda\wh A-I_\sK)^{-1}V\\[3mm]
\qquad\qquad =-D^2_{V}+\cfrac{2\lambda}{1-\lambda}V^*\left(I_\sK+\cfrac{1+\lambda}{1-\lambda}\,F\right)^{-1}V,\;\lambda\in\dC\setminus\dR_+
\end{array}
\end{equation}
belongs to the inverse Stieltjes class $\wt\cS^{-1}(\sM)$. In addition, one has
$\stackrel{0}{\cQ}_{\wh A,V}(-1)=I_\sM$, $\stackrel{0}{\cR}_{\wh A, V}(-1)=-I_\sM$.

The inverse Stieltjes class function $-(\stackrel {0}{\cQ}_{\wh A,V}(\lambda))^{-1}$ takes the form
\begin{equation} \label{jhhfn1}
\begin{array}{l}
-(\stackrel {0}{\cQ}_{\wh A,V}(\lambda))^{-1}=-I_\sM+V^*\left(\cfrac{2}{1+\lambda}\,I_\sK-(I_\sK+F)D^2_{V^*}\right)^{-1}(I_\sK+F)V \\
\qquad =-I_\sM+(1+\lambda)V^*\left(\wh A-\lambda I_\sK+(1+\lambda)VV^*\right)^{-1}V,\quad
 \lambda\in\dC\setminus\dR_+.
\end{array}
\end{equation}
\end{proposition}

\begin{proof} 
Define
\[
 N=\cfrac{1}{\sqrt{2}}(I_\sK+F)^{\half} V=(I_\sK+\wh A)^{-\half}V.
\]
Since $V$ is a contraction one has
\[
I_\sK+F\ge 2 NN^*\Longleftrightarrow (I_\sK+\wh A)^{-1}\ge NN^*.
\]
Hence by Proposition \ref{stu10} the operator-valued function
\[
 \cM_{0,1}(\lambda)=I_\sM+2\,\cfrac{1+\lambda}{1-\lambda}N^*\left(I_\sK-\cfrac{1+\lambda}{1-\lambda}F\right)^{-1}N,
 \quad \lambda\in\dC\setminus\dR_+,
\]
belongs to Stieltjes class $\wt\cS(\sM)$. By substituting the formula for $N$ one obtains
\begin{equation}\label{Q0}
\begin{array}{l}
 \cM_{0,1}(\lambda)
 =I_\sM+\,\cfrac{1+\lambda}{1-\lambda}V^*\left(I_\sK-\cfrac{1+\lambda}{1-\lambda}F\right)^{-1}(I+F)V\\
 \qquad=D^2_{V}+\cfrac{2}{1-\lambda}V^*\left(I_\sK-\cfrac{1+\lambda}{1-\lambda}\,F\right)^{-1}V\\
 \qquad=\; \stackrel{0}{\cQ}_{\wh A,V}(\lambda), \quad \lambda\in\dC\setminus\dR_+.
 \end{array}
\end{equation}
On the other hand, \eqref{connectresab} shows that
\[
 (\wh A-\lambda I_\sK)^{-1}=\cfrac{1}{1-\lambda}\left(I_\sK-\cfrac{1+\lambda}{1-\lambda}F\right)^{-1}(I_\sK+F),\quad \lambda\in\rho(\wh A),
\]
which yields the first formula in \eqref{dfcf1}:
\[
\stackrel{0}{\cQ}_{\wh A,V}(\lambda)=I_\sM+(1+\lambda)V^*(\wh A-\lambda I_\sK)^{-1}V,\quad \lambda\in\dC\setminus\dR_+.
\]

The formulas in \eqref{dfcf2} are obtained directly from \eqref{dfcf1}.

It remains to prove \eqref{jhhfn1}. To see this introduce the operator
$$G:=F-2NN^*=F-(I_\sK+F)^{\half}VV^*(I_\sK+F)^{\half}=(I_\sK+F)^{\half}D^2_{V^*}(I_\sK+F)^{\half}-I_\sK.$$
Then $G$ is a selfadjoint contraction.
Taking into account that the function $z=\frac{1+\lambda}{1-\lambda}$ maps the domain $\dC\setminus\dR_+$ onto the domain $\dC\setminus\{(-\infty,-1]\cup [1,\infty)\}$
an application of the equality \eqref{stuttg121} with a straightforward calculation leads to
\[
\begin{array}{l}
-(\stackrel{0}{\cQ}_{\wh A,V}(\lambda))^{-1}+I_\sM
 =2\,\frac{1+\lambda}{1-\lambda}N^*\left(I_\sK-\frac{1+\lambda}{1-\lambda}G\right)^{-1}N\\
=\frac{1+\lambda}{1-\lambda}V^*(I_\sK+F)^{\half}\left(I_\sK-\frac{1+\lambda}{1-\lambda}
    \left((I_\sK+F)^{\half}D^2_{V^*}(I_\sK+F)^{\half}-I_\sK\right)\right)^{-1}(I_\sK+F)^{\half}V\\
=\frac{1+\lambda}{1-\lambda}V^*(I_\sK+F)^{\half}
    \left(\frac{2}{1-\lambda}I_\sK-\frac{1+\lambda}{1-\lambda}(I_\sK+F)^{\half}D^2_{V^*}(I_\sK+F)^{\half}\right)^{-1}(I_\sK+F)^{\half}V\\
=\frac{1+\lambda}{1-\lambda}V^*\left(\frac{2}{1-\lambda}I_\sK-\frac{1+\lambda}{1-\lambda}(I_\sK+F)D^2_{V^*}\right)^{-1}(I_\sK+F)V\\
=\frac{1+\lambda}{2}V^*\left(I_\sK-\frac{1+\lambda}{2}(I_\sK+F)D^2_{V^*}\right)^{-1}(I_\sK+F)V\\
=(1+\lambda)V^*\left(I_\sK-(1+\lambda)(I_\sK+\wh A)^{-1}D^2_{V^*}\right)^{-1}(I_\sK+\wh A)^{-1}V\\
=(1+\lambda)V^*\left(\wh A-\lambda I_\sK+(1+\lambda){VV^*}\right)^{-1}V,\;\lambda\in\dC\setminus\dR_+,
\end{array}
\]
where we have used the identity $P(I-QP)^{-1}=(I-PQ)^{-1}P$ which holds whenever one of the inverses exists.
\end{proof}

\begin{corollary}\label{inverse11}
Let the functions $\stackrel{0}{\cQ}_{\wh A,V}$ and $\stackrel{0}{\cR}_{\wh A,V}$ be given by \eqref{dfcf1} and \eqref{dfcf2}, respectively.
Then for each $\lambda\in\dC\setminus \dR_+$ there are positive numbers $c_0(\lambda)$ and $d_0(\lambda)$ such that \begin{equation}
\label{stu1222}
\left|\left(\stackrel{0}{\cQ}_{\wh A,V}(\lambda)f,f\right)\right|\ge c_0(\lambda)\|f\|^2,\quad
\left|\left(\stackrel{0}{\cR}_{\wh A,V}(\lambda)f,f\right)\right|\ge d_0(\lambda)\|f\|^2\;\;\forall f\in\sM.
\end{equation}
\end{corollary}

\begin{corollary}\label{etwasneu}
Let $\wh A$ be a nonnegative selfadjoint relation in $\sK$ and let $V\in\bB(\sM,\sK)$ be a contraction.
Let $\wh A_o$ be the operator part of $\wh A$ and let $\wh P_o$ be the orthogonal projection onto the subspace
$\wh\sK_o=\sK\ominus\mul\wh A\,(=\cdom \wh A_o)$. Denote by $\stackrel{0}{P_{o}}$ the orthogonal
projection in $\sK$ onto $\cran \wh A_0$.
Then the $\bB(\sM)$-valued Stieltjes class function $\stackrel{0}{\cQ}_{\wh A,V}$ \eqref{dfcf1} takes the form
\begin{equation}\label{etwasneu2}
\stackrel{0}{\cQ}_{\wh A,V}(\lambda) 
 =I_{\sM}-V^*\wh P_oV+V^*\left((I+\wh A_o)(\wh A_o-\lambda I_{\wh\sK_o})^{-1}\wh P_o\right)V,\;\lambda\in\dC\setminus\dR_+.
\end{equation}
The $\bB(\sM)$-valued inverse Stieltjes function
\begin{equation}\label{etwasneu3}
\stackrel{0}{\cR}_{\wh A^{-1},V} (\lambda)=-I_\sM-(1+\lambda)V^*(\lambda\wh A^{-1}-I_\sK)^{-1}V,\;\lambda\in\dC\setminus\dR_+
\end{equation}
can be rewritten as follows
\begin{multline}\label{etwasneu4}
\stackrel{0}{\cR}_{\wh A^{-1},V}=-I_{\sM}+V^*P^\perp_o V+\\
\lambda V^*P^\perp_o V+V^*\wh A_o(I_{\wh\sK_o}+\wh A_o)\left((\wh A_o-\lambda I_{\wh\sK_o})^{-1}-(I_{\wh\sK_o}+\wh A_o)^{-1}\right)\stackrel{0}{P_{o}} V.
\end{multline}
\end{corollary}
\begin{proof}
From \eqref{dfcf1}  for each $\lambda\in\dC\setminus\dR_+$ one gets
\[
\begin{array}{l}
\stackrel{0}{\cQ}_{\wh A,V}(\lambda)=I_\sM+(1+\lambda)V^*(\wh A-\lambda I_\sK)^{-1}V\\
\qquad =I_{\sM}+V^*\left(P_o+P^\perp_o+(1+\lambda)(\wh A_o-\lambda I_\sK)^{-1}P_o\right)V-V^*V\\
\qquad =I_{\sM}-V^*P_oV+V^*\left((I_{\wh\sK_o}+\wh A_o)(\wh A_o-\lambda I_{\wh\sK_o})^{-1}P_o\right)V.
\end{array}
\]
So, we got \eqref{etwasneu2}.

Next observe that
\[
\begin{array}{l}
(I_\sK-\lambda \wh A^{-1})^{-1}=
I_\sK+\lambda(\wh A-\lambda I_\sK)^{-1}=\wh A_o(\wh A_o-\lambda I_{\wh\sK_o})^{-1}P_0+P^\perp_0\\
\qquad\qquad=\wh A_o(\wh A_o-\lambda I_{\wh\sK_o})^{-1}\stackrel{0}{P_{o}}+P^\perp_0.
\end{array}
\]
This leads to
\begin{multline*}
\stackrel{0}{\cR}_{\wh A^{-1},V}(\lambda)=-I_\sM-(1+\lambda)V^*(\lambda\wh A^{-1}-I_\sK)^{-1}V\\
=-I_{\sM}+(1+\lambda)V^*\left(I_\sK+\lambda(\wh A-\lambda I_\sK)^{-1}\right)V=
-I_{\sM}+V^*P^\perp_o V+\lambda V^*P^\perp_o V \\
+(1+\lambda)V^*\wh A_o(\wh A_o-\lambda I_{\wh\sK_o})^{-1}\stackrel{0}{P_{o}} V\\
=-I_{\sM}+V^*P^\perp_o V+\lambda V^*P^\perp_o V\\
    +(1+\lambda)V^*\wh A_o(I_{\wh\sK_o}+\wh A_o)(\wh A_o-\lambda I_{\wh\sK_o})^{-1}(I_{\wh\sK_o}+\wh A_o)^{-1}\stackrel{0}{P_{o}} V\\
=-I_{\sM}+V^*P^\perp_o V+\lambda V^*P^\perp_o V\\
+V^*\wh A_o(I_{\wh\sK_o}
     +\wh A_o)\left((\wh A_o-\lambda I_{\wh\sK_o})^{-1}-(I_{\wh\sK_o}+\wh A_o)^{-1}\right)\stackrel{0}{P_{o}} V.
\end{multline*}
\end{proof}

\begin{remark}
Given a selfadjoint contraction $F$ in $\sK$ and a contraction $V\in\bB(\sM,\sK)$ one can define the operator $N=\cfrac{1}{\sqrt{2}}(I_\sK-F)^{\half}$.
By Proposition \ref{stu10} the formula \eqref{cm2} shows that
\[
\wt\cS^{-1}(\sM)\ni\cM_{02}(\lambda)=-D^2_V+\cfrac{2\lambda}{1-\lambda}V^*\left(I_\sK-\cfrac{1+\lambda}{1-\lambda}F\right)^{-1}V,\;\; \lambda\in\dC\setminus\dR_+.
\]
Then
\[
\wt\cS(\sM)\ni-\cM_{02}(1/\lambda)=D^2_V+\cfrac{2}{1-\lambda}V^*\left(I_\sK+\cfrac{1+\lambda}{1-\lambda}F\right)^{-1}V,\;\; \lambda\in\dC\setminus\dR_+
\]
and
\[
\wt\cS(\sM)\ni-(\cM_{02}(\lambda))^{-1}
=I_\sM+V^*\left((I_\sK-F)D^2_{V^*}-\frac{2\lambda}{1+\lambda}I_\sK\right)^{-1}(I_\sK-F)V,\;\; \lambda\in\dC\setminus\dR_+.
\]
So, if $F$ is replaced by $-F$ (equivalently $\wh A=\cC(F)$ is replaced by $\wh A^{-1}=\cC(-F)$) in \eqref{dfcf1}, \eqref{dfcf2} and  \eqref{jhhfn1}, then in fact
\[
\begin{array}{l}
\stackrel{0}{\cQ}_{\wh A^{-1},V}(\lambda)=-\cM_{02}(1/\lambda),\quad
\stackrel{0}{\cR}_{\wh A^{-1}, V}(\lambda)=\cM_{02}(\lambda),\\[3mm]
-\stackrel{0}{\cQ}_{\wh A^{-1},V}(\lambda)^{-1}=(\cM_{02}(1/\lambda))^{-1}
\;\; \forall\lambda\in\dC\setminus\dR_+.
\end{array}
\]
\end{remark}

\subsection{Stieltjes/inverse Stieltjes functions whose values are unbounded operators/l.r.'s in $\bB(\sM)$}

Recall from \cite{Ka} that: 1) a family $\cV(\lambda)$ of closed linear operators or l.r.'s defined on a domain $\Pi\subset\dC$, is called \textit{holomorphic of the type(A)} in $\Pi$ if $\dom\cV(\lambda)\equiv\cD $ (does not depend on $\lambda\in\Pi$) and the vector-function $\cV(\lambda)f$ is holomorphic on $\Pi$ for each $f\in\cD$:
2) a family $\cV(\lambda)$, $\lambda\in \Pi$  of $m$-sectorial l.r.'s is called \textit{holomorphic of the type(B)} in $\Pi$ if $\cD[\cV(\lambda)]\equiv\cD$
(domain of the closed associated form does not depend on $\lambda\in\Pi$) and $\cV(\lambda)[f,g]$ is a holomorphic scalar function for every $f,g\in\cD$.

\begin{proposition}
\label{stutcor1}
Let $F$ be a selfadjoint contraction in the Hilbert space $\sK$, let $\wh A=\cC(F)$ and let $V\in\bB(\sM,\sK)$ be a contraction. Let $\stackrel{0}{\cQ}_{\wh A,V}$ and $\stackrel{0}{\cR}_{\wh A,V}$ be given by \eqref{dfcf1} and \eqref{dfcf2}. Assume that $\cZ$ is a closed (not necessary densely defined) linear operator in $\sM$.
Then for each $\lambda\in\dC\setminus\dR_+$ the sesquilinear forms
\begin{equation}\label{theforms}
\begin{array}{l}
\wh{\cQ}_{\wh A,V,\cZ}(\lambda)[f,g]=\left(\stackrel{0}{\cQ}_{\wh A,V}(\lambda)\cZ f,\cZ g\right),\\
\wh\cR_{\wh A,V,\cZ}(\lambda)[f,g]=\left(\stackrel{0}{\cR}_{\wh A,V}(\lambda)\cZ f,\cZ g\right),\; f,g\in\dom\cZ
\end{array}
\end{equation}
are closed and sectorial. Besides, $\wh {\cQ}_{\wh A,V}(\lambda)[f,g]$ and $ \wh \cR_{\wh A,V}(\lambda)[f,g]$ are holomorphic on $\dC\setminus\dR_+$ for all $f,g\in\dom\cZ$. The associated l.r.'s $\wh{\cQ}_{\wh A,V,\cZ}(\lambda)$ and $\wh\cR_{\wh A,V,\cZ}(\lambda)$ are given by
\begin{equation}
\label{therelat}
\begin{array}{l}
\left\{\begin{array}{l}\dom\wh{\cQ}_{\wh A,V,\cZ}(\lambda)=\left\{f\in\dom \cZ:\stackrel{0}{\cQ}_{\wh A,V}(\lambda)\cZ f\in\dom \cZ^*\right\},\\
\wh{\cQ}_{\wh A,V,\cZ}(\lambda)f=\cZ^*\stackrel{0}{\cQ}_{\wh A,V}(\lambda)\cZ f,\;f\in\dom\wh\cQ(\lambda)
\end{array}\right.\\[4mm]
\left\{\begin{array}{l}\dom\wh\cR_{\wh A,V,\cZ}(\lambda)=\left\{f\in\dom \cZ:\cR_0(\lambda)\cZ f\in\dom \cZ^*\right\},\\
\wh\cR_{\wh A,V,\cZ}(\lambda)f=\cZ^*\stackrel{0}{\cR}_{\wh A,V}(\lambda)\cZ f,\;f\in\dom\wh\cR(\lambda).
\end{array}\right.
\end{array}
\end{equation}
Thus, $\wh {\cQ}_{\wh A,V,\cZ}$ and $\wh\cR_{\wh A,V,\cZ}$ form holomorphic families of the type (B). Moreover, $\wh{\cQ}_{\wh A,V,\cZ}$ is a Stieltjes family, while $\wh\cR_{\wh A,V,\cZ}$ is an inverse Stieltjes family.
\end{proposition}
\begin{proof}
The bounded operator-valued function $\stackrel{0}{\cQ}_{\wh A,V}(\lambda)$ belongs to the Stieltjes class.
It follows that $\stackrel{0}{\cQ}_{\wh A,V}(\lambda)$ are bounded sectorial operators for each $\lambda\in\dC\setminus\dR_+$; see Theorem~\ref{cbvgfn}.
Let us show that the form $\stackrel{0}{\cQ}_{\wh A,V}(\lambda)[f,g]$ with the domain $\cD[\stackrel{0}{\cQ}_{\wh A,V}(\lambda)]=\dom \cZ$ is closed in $\sM.$
If $\{f_n\}\subset\dom\cZ$ and
$$\lim_{n\to\infty}f_n=f,\; \lim_{n,m\to\infty}\left|\left(\stackrel{0}{\cQ}_{\wh A,V}(\lambda)\cZ(f_n-f_m),\cZ(f_n-f_m)\right)\right|=0,$$
then \eqref{stu1222} in Corollary \ref{inverse11} shows that $ \lim_{n,m\to\infty}\left\|\cZ(f_n-f_m)\right\|=0$. Because $\cZ$ is a closed operator, one has
\[
f\in\dom\cZ,\quad \lim_{n\to \infty}\cZ f_n=\cZ f.
\]
Now the boundedness of $\stackrel{0}{\cQ}_{\wh A,V}(\lambda)$ yields
\[
\lim_{n\to\infty}\left(\stackrel{0}{\cQ}_{\wh A,V}(\lambda)\cZ(f-f_n),\cZ(f-f_n)\right)=0,\quad
\lim_{n\to\infty}\stackrel{0}{\cQ}_{\wh A,V}(\lambda)[f_n]=\stackrel{0}{\cQ}_{\wh A,V}(\lambda)[f].
\]
Thus, the sesquilinear form $\wh{\cQ}_{\wh A,V,\cZ}(\lambda)[\cdot,\cdot]$ given by \eqref{theforms} is closed in $\sM$.
By the first representation theorem \cite{Ka}, \cite{RoBe} the associated l.r. is given by \eqref{therelat}.

Similar proof can be given for $\wh\cR_{\wh A,V,\cZ}.$
\end{proof}

\subsection{An example of a special type of Stieljes/inverse Stieltjes function/family}

Here we construct an example of a Stieltjes/inverse Stieltjes family $\cQ$ in $\sM$ such that
$$\dom\cQ(\lambda)\cap\dom\cQ(\mu)=\{0\},\;\forall \lambda,\mu,\; \lambda\ne\mu.$$
This next example is analogous to \cite[Example~4.4]{DHM15}, where the existence of Nevanlinna functions
$F(\lambda)$, $\lambda\in\cmr$, whose values are densely defined unbounded operators
on a Hilbert space $\cH$ such that
\[
 \dom F(\lambda) \cap \dom F(\mu)=\{0\}, \quad \forall \, \lambda\neq \mu, \; \lambda,\mu \in \cmr,
\]
while the function $F(\lambda)$ is form domain invariant in the sense that the form
\begin{equation}\label{ImF}
 \st_{F(z)}[u,v]:=\frac{1}{z-\bar z}\,[(F(z)u,v)-(u,F(z)v)], \quad u,v\in \dom F(z),
\end{equation}
is closable and the closure has a constant domain in $\cH$.


\begin{example}
Assume $\sM$ is an infinite-dimensional separable Hilbert space.
Let $\wh A$ be the zero-operator in an auxiliary
infinite-dimensional Hilbert space $\sK$, i.e., $\dom\wh A=\sK$,
$\wh Af=0$ for all $f\in\sK$. Then for an arbitrary contraction
$V\in\bB(\sM,\sK)$ the $\bB(\sM)$-valued function $\cQ_{01}(\wh A,
V,\lambda)=I_\sM+(1+\lambda)V^*(\wh A-\lambda I)^{-1}V$  
takes the form
$$\cQ_{01}(0,V,\lambda)=I_\sM+(1+\lambda)(-\lambda^{-1})V^*V=I_\sM-(\lambda^{-1}+1)V^*V.$$
The function $\cQ_{01}(0,V,\lambda)$ belongs to $\wt S(\sM)$, on the domain $\dC\setminus\dR_+$ and $\cQ_{01}(0,V,-1)=I_\sM$.
Now suppose $\ran V^*V\ne \sM,$ $\ker V=\{0\}$. Then the range of the operator $V^*V$ is a dense linear manifold in $\sM$. Due to the von Neumann theorem \cite{FW} one can find a bounded nonnegative selfadjoint operator $X$ such that
$\ker X=\{0\},$ $\ran X\cap\ran V^*V=\{0\}$. Set $\cZ:=X^{-1}$ and let
\begin{multline*}
\cQ_1(\lambda)=\cZ\cQ_{01}(0,V,\lambda)\cZ=\cZ\left(I_\sM-(\lambda^{-1}+1)V^*V\right)\cZ,\\
\dom\cQ_1(\lambda)=\left\{f\in\dom\cZ:\left(I_\sM-(\lambda^{-1}+1)V^*V\right)\cZ f\in\dom\cZ\right\},
\lambda\in\dC\setminus\dR_+.
\end{multline*}
Then $\dom\cQ_1(\lambda)$ is dense in $\sM$ for all $\lambda\in\dC\setminus \dR_+$.
Due to the condition $\dom Z\cap\ran V^*V=\{0\}$ we get that $\dom \cQ_1(\lambda)\cap\dom Z^2=\{0\}$ for all $\lambda\in\dC\setminus\dR_+$.
Let $\lambda_1,\lambda_2\in\dC\setminus\dR_+$, $\lambda,\mu$. Suppose $f\in\dom \cQ_1(\lambda)\cap\dom \cQ_1(\mu)$. Then
\[
\left(I_\sM-(\lambda^{-1}+1)V^*V\right)\cZ f, \left(I_\sM-(\mu^{-1}+1)V^*V\right)\cZ f\in\dom\cZ.
\]
It follows that
\begin{multline*}
\left(I_\sM-(\lambda^{-1}+1)V^*V\right)\cZ f-\left(I_\sM-(\mu^{-1}+1)V^*V\right)\cZ f\in\dom\cZ\\
\Longrightarrow V^*V\cZ f\in\dom \cZ\iff V^*V\cZ f\in\ran X\Longrightarrow \cZ f=0\Longrightarrow f=0.
\end{multline*}
Thus, $\dom\cQ_1(\lambda)\cap\dom\cQ_1(\mu)=\{0\}$. Observe that $\dom\cZ^2=\dom\cQ_1(-1).$

Set  $\cQ_2(\lambda):=-\cQ_1(\lambda^{-1})$. Then $\cQ_2\in \wt S^{-1}(\sM)$ and, clearly,  $\dom\cQ_2(\lambda)\cap\dom\cQ_2(\mu)=\{0\}$ for
all distinct $\lambda,\mu\in\dC\setminus\dR_+$.

Arguing similarly, we get that if $\ran V^*\cap\ran \cZ =\{0\}$ ($\cZ$ is unbounded selfadjont operator), then for an arbitrary $\wh A$ the corresponding Stieltjes and inverse Stieltjes families
\[\begin{array}{l}
Q_{1}(\lambda)=\cZ\left(I_\sM+(1+\lambda)V^*(\wh A-\lambda I)^{-1}V\right)\cZ,\\
Q_{2}(\lambda)=\cZ\left(-I_\sM-\left(1+\frac{1}{\lambda}\right)V^*\left(\wh A-\frac{1}{\lambda} I\right)^{-1}V\right)\cZ.$$
\end{array}
\]
possess the properties
\[
\dom Q_{1}(\lambda)\cap\dom Q_1(\mu)=\{0\},\;\dom Q_{2}(\lambda)\cap\dom Q_2(\mu)=\{0\},
\]
for all $\lambda,\mu\in\dC\setminus\dR_+,\;\lambda\ne\mu.$
\end{example}

Form domain invariant Nevanlinna functions have been introduced and studied extensively in \cite{DM97,DHM15,DHM20a,DHM20b},
see also \cite{DHM17}.
It is interesting to note that in \cite[Example~6.7]{DHM17} (or \cite[Example~2.7]{DHM20b})
it is shown that there are Stieltjes and inverse Stieltjes functions
whose imaginary parts as defined in \eqref{ImF} need not be form domain invariant,
while according to Theorem \ref{typB} (as proved in the next section) the closed sectorial forms associated with
these functions themselves have a constant domain.

\section{Closed sesquilinear forms associated with Stieltjes and inverse Stieltjes families}\label{sesqu}

The next theorem is the main result in this section.
\begin{theorem}
\label{typB}
(1) If a family belongs to Stieltjes/inverse Stieltjes class, then it is a holomorphic family of type (B)
in the sense of \cite{Ka}.

(2) Let $\cQ(\lambda)$ be a Stieltjes family in $\sM$. Then there
exist a Hilbert space $\sK'$, a nonnegative selfadjoint relation
$A'$ in $\sK'$, a contraction $N'\in\bB(\sM,\sK')$, and
a closed linear operator $\cX$ in $\sM$ with $\dom\cX=\cD[\cQ(-1)]$ such that
$(I+A')^{-1}\ge N'N'^*$ and for all $u,v\in\dom\cX$,
\[
\begin{array}{l}
\cQ(\lambda)[u,v]
=\left(\left(I_{\sM} +(1+\lambda)N'^*\left(I+(1+\lambda)( A'-\lambda I)^{-1}\right)N'\right)\cX u, \cX v\right).
\end{array}
\]
Therefore there is a contraction $V'\in\bB(\sM,\sK')$ such that for all $\lambda\in\dC\setminus\dR_+$,
\begin{equation}\label{QQQ}
\cQ(\lambda)[u,v]
=\Biggl(\Biggl(I_\sM+(1+\lambda)V'^*( A'-\lambda I_\sK)^{-1}V'\Biggr)\cX u, \cX v\Biggr),\quad  u,v\in\cD[\cQ(-1)].
\end{equation}

(3) Let $\cR(\lambda)$ be an inverse Stieltjes family in $\sM$. Then there exist a Hilbert space $\sK''$, a nonnegative selfadjoint
relation $A''$ in $\sK''$, a contraction $N''\in\bB(\sM,\sK)$, and a closed linear operator $\cY$ in $\sM$ with $\dom\cY=\cD[\cR(-1)]$
such that $(I+ A'')^{-1}\ge N''N''^*$ and for all $h,g\in \dom\cY$,
\[
\begin{array}{l}
\cR(\lambda)[h,g]
 =\left(\left(-I_{\sM} +(1+\lambda)N''^*\left(I+(1+\lambda)(A''-\lambda I)^{-1}\right)N''\right)\cY h, \cY g\right).
\end{array}
\]
Therefore there is a contraction $V''\in \bB(\sM,\sK'')$ such that for all $\lambda\in\dC\setminus\dR_+$,
\[
 \cR(\lambda)[h,g]=\Biggl(\Biggl(-I_\sM+(1+\lambda)V''^*
 (I_\sK-\lambda (A'')^{-1})^{-1}V''\Biggr)\cY h, \cY g\Biggr),\quad  h,g\in \cD[\cR(-1)].
\]

(4) If $\cR(\lambda)=-\cQ(\lambda)^{-1}$ then one can choose $\sK:=\sK_1=\sK_2$ and $N:=N'=N''$ such that
\[
 NN^*=(I+ A')^{-1}-(I+A'')^{-1}.
\]
\end{theorem}

\begin{proof}
First the case of a Stieltjes family $\cQ(\lambda)$ is considered.

By Lemma \ref{TH1} there is $\Omega(z)\in\cRS(\sM)$ such that
\begin{equation}\label{Omega0}
 \cQ(\lambda)=-I+2\left(I-\Omega\left(\cfrac{1+\lambda}{1-\lambda}\right)\right)^{-1}.
\end{equation}
Denote $z=\frac{1+\lambda}{1-\lambda}$.
We proceed by constructing a more explicit representation for $I-\Omega(z)$ by rewriting $\Omega(z)$
as a transfer function
\[
 \Omega(z)=D+zC(I-zF)^{-1}C^*,\; z\in\dC\setminus\{(-\infty,-1]\cup[1,+\infty)\},
\]
of a passive selfadjoint system
$\tau=\left\{ \begin{bmatrix} D&C\cr C^*&F\end{bmatrix},\sM,\sM,\sK\right\}$
with input-output space $\sM$ and the state space $\sK$.
Then the entries of the selfadjoint contraction
\[
T=\begin{bmatrix} D&C\cr
C^*&F\end{bmatrix}:\begin{array}{l}\sM\\\oplus\\
\sK\end{array}\to
\begin{array}{l}\sM\\\oplus\\
\sK\end{array}
\]
take the form \eqref{parame}. Note that $D=\Omega(0)$.
It follows that for all $z\in\dC\setminus\{(-\infty,-1]\cup[1,+\infty)\}$
\[
\begin{array}{l}
 I-\Omega(z) =I-D-zC(I-zF)^{-1}C^*\\
  =(I-D)^{\half}\left(I-z(I+D)^{\half}N^*(I-zF)^{-1}NP_{\sD_D}(I+D)^{\half}\right)(I-D)^{\half}\\
  =(I-\Omega(0))^{\half}(I-\Sigma_+(z)) (I-\Omega(0))^{\half},
\end{array}
\]
where $\Sigma_+(z)=z(I+D)^{\half}N^*(I-zF)^{-1}NP_{\sD_D}(I+D)^{\half}$
is defined in \eqref{zzz}.

 Let a selfadjoint contraction $F'$ be defined as in \eqref{ff}, i.e.,
\[
F'=NN^*+D_{N^*}XD_{N^*}.
\]
By Lemma \ref{LEM1} the bounded inverse $(I-\Sigma_+(z))^{-1}$ for all $z\in\dC\setminus\{(-\infty,-1]\cup[1,+\infty)\}$
is given by
\begin{equation}\label{Omega2}
 (I-\Sigma_+(z))^{-1}=I+z(I+D)^{\half}N^*(I-zF')^{-1}NP_{\sD_D}(I+D)^{\half}.
\end{equation}

Suppose $\RE\lambda<0$. 
  Then $|z|<1$ and $\Omega(z)$ is a contraction, $\Omega(z)^*=\Omega(\bar z)$,
$\Omega(x)$ is selfadjoint contraction for all $x\in(-1,1)$. Moreover, the operator $\Omega(z)$ belongs to the class $\wt C_\sM(\alpha_z),$ $\alpha_z=\arctan\left(\cfrac{2|\IM z|}{1-|z|^2}\right),$ see \eqref{omca}. Therefore, the bounded operator $I-\Omega(z)$ and the l.r. $(I-\Omega(z))^{-1}$ are $m-\alpha_z$-sectorial. By Proposition \ref{clfrm} the domain $\cD[(I-\Omega(z))^{-1}]$ of the closed form associated with $(I-\Omega(z))^{-1}$  coincides with $\ran(I-\RE\Omega(z))^{\half}$. Since
$I-\RE\Omega(z)$ is harmonic function, the Harnack inequalities yields the equality (see  \cite{Shmul2}):
\begin{equation}
\label{ravenstva1}
\ran(I-\RE\Omega(z))^{\half}=\ran(I-\Omega(0))^{\half},\;
 z\in\dD.
\end{equation}
Thus we have the equality
\[
\cD[\cQ(\lambda)]=\cD[\cQ(-1)],\;\RE \lambda<0.
\]
From the above expressions of $I-\Omega(z)$ it follows that the bounded operators $I-\Sigma_+(z)$ and $(I-\Sigma_+(z))^{-1}$ are sectorial and
\begin{multline*}
((I-\Omega(z))^{[-1]} f,f)=\left((I-\Sigma_+(z))^{-1} (I-\Omega(0))^{[-\half]}f,(I-\Omega(0))^{[-\half]}f\right),\\
 f\in\ran(I-\Omega(z))=\dom\cQ(\lambda).
\end{multline*}

Now rewrite
\[
\begin{array}{c}
 u=(I-\Omega(0))^{\half}(I-\Omega(0))^{[-\half]}u,\\
 ||u||^2=((I-\Omega(0))(I-\Omega(0))^{[-\half]}u, (I-\Omega(0))^{[-\half]}u),
\end{array}
\]
and denote
\[
\cX=(I+\Omega(0))^{\half}(I-\Omega(0))^{[-\half]}.
\]
Then it follows from \eqref{Omega0}, \eqref{Omega2}, and Proposition \ref{stutcor1} that the closed quadratic form $\cQ[\cdot]$ associated with $\cQ(\lambda)$ admits the expression
\[
\cQ(\lambda)[u]
=\left(\left(I-2N^*\left(F'-\frac{1-\lambda}{1+\lambda}I\right)^{-1}N\right)\cX u,\cX u \right),\;
u\in\dom \cX=\cD[\cQ(-1)].
\]
Let the nonnegative selfadjoint l.r. $A'=-I+2(I+F')^{-1}$ be the Cayley transform of $F'$. Then $N=(A'+I)^{-\half} V'$ for some contraction
$ V'\in\bB(\sM,\sK)$ and the expression for $\cQ(\lambda)[u]$ can be rewritten as follows
\[
\cQ(\lambda)[u]=\left(\left(I_\sM+(1+\lambda)V'^*(A'-\lambda I_\sK)^{-1}V'\right)\cX u,\cX u\right),\; u\in\dom\cX,\; \RE\lambda<0.
\]
Thus, $\cQ(\lambda)[\cdot,\cdot]$ is of the form $\wh{\cQ}_{\wh A,V,\cZ}(\lambda)[\cdot,\cdot]$ given by \eqref{theforms};
see also \eqref{dfcf1}, \eqref{Q0}, Proposition \ref{stutcor1}.
Besides, the function $\cQ(\lambda)[f,g]$ is holomorphic on $\dC\setminus\dR_+$.
This means that $\cQ(\lambda)$ is a holomorphic family of type (B).

By means of the transform $R(\lambda)=-\cQ\left(\lambda^{-1}\right)$ one concludes that also
inverse Stieltjes families are holomorphic families of type (B). Thus, assertion (1) is proven and we have representations of $\cQ(\lambda)[u,v]$, $u,v\in\dom\cX=\cD[\cQ(-1)]$, in the statement (2).

Similarly, using the representation
$$\cR(\lambda)=-\cQ^{-1}(\lambda)= I-2(I+\Omega(z))^{-1},\; z=\cfrac{1+\lambda}{1-\lambda},\;\lambda\in\dC\setminus\dR_+$$
for the holomorphic inverse Stieltjes family $\cR(\lambda)$ in $\sM$ and  the contraction $F''=-NN^*+D_{N^*}XD_{N^*}$, one can prove that $\cR(\lambda)$ is  of type (B), the sesquilinear form $\cR[\cdot,\cdot]$ admits representations in the statement (3), and if $\cR(\lambda)=-\cQ(\lambda)^{-1}$, then the statement (4) is valid.

This completes the proof.
\end{proof}

Notice that the representations for $\cQ(\lambda)$ in \eqref{Q1}, $\cR(\lambda)$ \eqref{R1}, and for the corresponding
closed forms as stated in item (2) of Introduction follow from \eqref{QQQ} with $\cZ=\cX=(I+\Omega(0))^{\half}(I-\Omega(0))^{[-\half]}$.

\begin{remark}
%
Using a slightly different approach another similar type of representation for the transformed family
$\cQ(\lambda)=I-2\left(I-\Omega\left(\frac{1+\lambda}{1-\lambda}\right)\right)^{-1}$
was constructed in \cite[Theorem~4.4]{ArlHassi_2015} in the case that $\lambda$ belongs to the left open half-place, $\RE \lambda<0$.
\end{remark}

\begin{corollary}
(1) Let $\cQ\in\wt\cS(\sM)$. Then the function
\[
\sK(\lambda,\mu)[\cdot,\cdot]:=\cQ(\lambda)[\cdot,\cdot]+\cQ(\bar\mu)[\cdot,\cdot]+
\cfrac{\lambda+\bar\mu}{\lambda-\bar\mu}\Biggl(\cQ(\lambda)[\cdot,\cdot]-\cQ(\bar\mu)[\cdot,\cdot]\Biggr)
\]
is a nonnegative kernel on the domain $\cD[\cQ(-1)]$.

 (2) Let $\cR\in\wt\cS^{-1}(\sM)$. Then the function
\[
\cL(\lambda,\mu)[\cdot,\cdot]:=\cR(\lambda)[\cdot,\cdot]+\cR(\bar\mu)[\cdot,\cdot]+
\cfrac{\lambda+\bar\mu}{\lambda-\bar\mu}\Biggl(\cR(\lambda)[\cdot,\cdot]-\cR(\bar\mu)[\cdot,\cdot]\Biggr)
\]
is a nonpositive kernel on the domain $\cD[\cR(-1)]$.
\end{corollary}

The next result, which gives integral representations for bounded operator-valued Stieltjes and inverse Stieltjes functions,
is well known in the scalar case, see \cite{KacK}. In the operator-valued case this result can be found in \cite[Appendix A6]{BHS2020}.
Here it is shown how these integral representations can be derived directly from the corresponding operator representations
in Theorem \ref{typB}.

\begin{theorem}\label{intrep1}
(1) Every $\bB(\sM)$-valued Stieltjes function $\cQ$ admits an integral representation of the form
\begin{equation}\label{interep2}
\cQ(\lambda)=\Gamma_\cQ+\int\limits_{\dR_+}\cfrac{d \Sigma_\cQ(t)}{t-\lambda},
\end{equation}
where $\Gamma_\cQ=\Gamma^*_\cQ\in\bB(\sM)$, $\Gamma_\cQ\ge 0$, and $\Sigma_\cQ(t)$ is a $\bB(\sM)$-valued non-decreasing function on $\dR_+$ such that $\Sigma_\cQ(0)=0$ and $\int\limits_{\dR_+}\cfrac{(d \Sigma_\cQ(t)f,f)}{t+1}<\infty$ for all $f\in\sM.$

(2) Every $\bB(\sM)$-valued inverse Stieltjes function $\cR$ admits an integral representation of the form
\begin{equation}\label{interep3}
\cR(\lambda)=\Gamma_\cR+\lambda \Pi_\cR+\int\limits_{\dR_+}\left(\cfrac{1}{t-\lambda}-\cfrac{1}{t}\right)d\Sigma_\cR(t),
\end{equation}
where $\Gamma_\cR=\Gamma^*_\cR\in\bB(\sM)$, $\Gamma_\cR\le 0,$ $\Pi^*_\cR=\Pi_\cR\in\bB(\sM),$ $\Pi_\cR\ge 0$, and $\Sigma_\cR(t)$ is a $\bB(\sM)$-valued non-decreasing function on $\dR_+$ such that $\Sigma_\cR(0)=0$ and $\int\limits_{\dR_+}\cfrac{(d\Sigma_\cR(t)f,f)}{t(t+1)}<\infty$ for all $f\in\sM.$
\end{theorem}
\begin{proof}
(1) By Theorem \ref{typB} the function $\cQ$ admits the operator representation
\[
\cQ(\lambda)
=\cZ^*\Biggl(I_\sM+(1+\lambda)V^*(\wh A-\lambda I_\sK)^{-1}V\Biggr)\cZ,\;\;\lambda\in\dC\setminus\dR_+,
\]
where $\wh A$ is a nonnegative selfadjoint l.r. in some Hilbert space $\sK$, $V\in\bB(\sM,\sK)$ is a contraction and $\cZ\in\bB(\sM).$

From Corollary \ref{etwasneu} we get
\[
\cQ(\lambda)
 =\cZ^*(I_{\sM}-V^*\wh P_oV)\cZ+\cZ^*V^*\left((I+\wh A_o)(\wh A_o-\lambda I_{\wh\sK_o})^{-1}\wh P_o\right)V\cZ,\quad
 \lambda\in\dC\setminus\dR_+,
\]
where $\wh A_o$ is the operator part of $\wh A$ and $\wh P_o$ is the orthogonal projection onto the subspace $\wh\sK_o=\sK\ominus\mul\wh A$.
Let $\wh E_o(t),$ $t\in \dR_+$,  be the resolution of identity of the operator part $\wh A_o$ in the subspace $\sK_o$. Define
\[
\begin{array}{l}
\Sigma_\cQ(t):=\int\limits_{0}^t(1+\tau)d\left(\cZ^*V^*\wh E_o(\tau)\wh P_oV\cZ \right)\Longrightarrow
 d\Sigma_\cQ(t)=(1+t)d\left(\cZ^*V^*\wh E_o(t)\wh P_oV\cZ \right),\\
 \Gamma_\cQ:=\cZ^*(I_{\sM}-V^*\wh P_oV)\cZ.
\end{array}
\]
This gives \eqref{interep2} with
\[
\Gamma_\cQ=\Gamma_\cQ^*\ge 0,\;\;\int\limits_{\dR_+}\cfrac{(d  \Sigma_\cQ(t)f,f)}{t+1}=||\wh{P_{o}}V\cZ f||^2\;\;\forall f\in\sM.
\]
(2) By Theorem \ref{typB} and the equalities \eqref{etwasneu3} and \eqref{etwasneu4} one can write
\[
\begin{array}{l}
\cR(\lambda)=\cY^*\biggl(-I_\sM-(1+\lambda)V^*(\lambda\wh A^{-1}-I_\sK)^{-1}V\biggr)\cY\\
\qquad=\cY^*\biggl(-I_{\sM}+V^*P^\perp_o V+\lambda V^*P^\perp_o V\biggr)\cY\\
\qquad + \cY^*\biggl(V^*\wh A_o(I_{\wh\sK_o}+\wh A_o)
  \left((\wh A_o-\lambda I_{\wh\sK_o})^{-1}-(I_{\wh\sK_o}+\wh A_o)^{-1}\right)\stackrel{0}{P_{o}} V\biggr)\cY,\;
    \quad\lambda\in\dC\setminus\dR_+,
\end{array}
\]
where $\stackrel{0}{P_{o}}$ is the orthogonal
projection in $\sK$ onto $\cran \wh A_0$.
Now define
\[
\begin{array}{l}
\Sigma_\cR(t):=\int_{0}^t\tau(1+\tau)d\left(\cY^*V^*\wh E_o(\tau)\stackrel{0}{P_{o}}V\cY\right)\\
\quad \Longrightarrow d\Sigma_\cR(t)=t(1+t)d\left(\cY^*V^*\wh E_o(t)\stackrel{0}{P_{o}}V\cY\right),\\
\quad \Gamma_\cR=\cY^*\left(-I_\sM+V^*({P^\perp_{o}}+\stackrel{0}{P_{o}})V\right)\cY,\; \textrm{ and } \Pi_\cR=\cY^*V^*{P^\perp_{o}}V\cY.
\end{array}
\]
Then $\Gamma_\cR\le 0$ and $ \Pi_\cR\ge 0$ are selfadjoint, and
\[
\int\limits_{\dR_+}\cfrac{d(\Sigma_\cR(t)f,f)}{t(t+1)}=||\stackrel{0}{P_{o}}V\cY f||^2\;\forall f\in\sM.
\]
This leads to \eqref{interep3}.
\end{proof}

\section{Limit values at $-\infty$ and $-0$}

The next result easily follows from \cite[Theorem 3.1, Theorem 3.2]{Simon1978}, \cite[Proposition 6]{DM1991}, \cite[Theorem 3.1]{BHSW2010} but we give an
independent proof in the present special situation.


\begin{lemma}\label{vspom}
Let $H$ be a Hilbert space and let $L(x)$, $x\in(a,b)$ be a function
whose values are bounded nonnegative selfadjoint operators acting on $H$.
Suppose that
\begin{enumerate}
\item $L(x)$ is non-decreasing on $(a,b)$,
\item $L(b):=s-\lim\limits_{x\uparrow  b} L(x)$ exists as a bounded operator,
\item $\ran L^{\half}(x)=\cR_0$ is constant for $x\in(a,b)$.
\end{enumerate}
Then
\begin{equation}
\label{nomerod}
 \lim\limits_{x\downarrow a} L^{-1}(x)[f]=\left\{\begin{array}{l}L^{-1}(a)[f],\;f\in\cD[L^{-1}(a)],\\
+\infty,\qquad  f\in\cR_0\setminus\cD[L^{-1}(a)],
\end{array}\right.
\end{equation}
\begin{equation}
\label{nomerdva}
 \lim\limits_{x\uparrow b} L^{-1}(x)[f]=L^{-1}(b)[f],\;f\in\cR_0 \subseteq \cD[L^{-1}(b)].
\end{equation}
\end{lemma}


\begin{proof}
By assumption (1) the strong limit value $L(a)$ exists as a bounded operator. Clearly.
$L(a)\le L(x)\le L(b)$ for each $x\in (a,b)$. Hence, $\ran
L^{\half}(a)\subseteq\cR_0 \subseteq\ran L^{\half}(b)$. The inverses $L^{-1}(x)$, $L^{-1}(a)$ and
$L^{-1}(b)$ are in general nonnegative selfadjoint l.r.'s
The corresponding closed sesquilinear forms are defined on $\cR_0$,
$\ran L^{\half}(a)$, and $\ran L^{\half}(b)$, respectively. Since
$L(x_1)\le L(x_2)$ if $x_1<x_2$, then $L^{-1}(x_1)\ge L^{-1}(x_2)$.
In addition $L^{-1}(a)\ge L^{-1}(x)\ge L^{-1}(b)$.
Let us prove that
\begin{equation}
 \label{domchar}
 \sup\limits_{x\in(a,b)}L^{-1}(x)[f]<\infty\iff f\in \cD[L^{-1}(a)]=\ran L^{\half}(a).
\end{equation}
If $f\in \cD[L^{-1}(a)]$, then $L^{-1}(x)[f]\le L^{-1}(a)[f]<\infty$.
Conversely, assume that
\[
 L^{-1}(x)[f]\le C\quad\mbox{for all}\quad x\in (a,b)\quad \mbox{and for some}\quad f\in \cD[L^{-1}(x)]=\cR_0.
\]
Then there exists a sequence of numbers $\{x_n\}\subset (a,b)$ such that the sequence of vectors $\{g_n:=L^{-\half}(x_n)f\}$
converges weakly to some vector $\f\in\overline\cR_0$, i.e.,
\[
 \lim\limits_{n\to\infty}(g_n, h)=(\f,h)\quad\mbox{for all}\quad h\in \overline\cR_0.
\]
Here $L^{-\half}(x)=\left(L^{\half}(x)\right)^{-1}$ is the Moore-Penrose pseudoinverse.
Since
$$
\lim\limits_{x\downarrow a}L(x)h=L(a)h \quad\Rightarrow\quad
\lim\limits_{x\downarrow a}L^{\half}(x)h=L^{\half}(a)h
$$
for all $h$, one gets
\[
\begin{array}{l}
(f,h)=(L^{\half}(x_n)g_n,h)=(g_n,L^{\half}(x_n)h) \\
\qquad\quad =(g_n,L^{\half}(a)h)+(g_n,(L^{\half}(x_n)-L^{\half}(a))h) \to(\f,L^{\half}(a)h)\; (n\to\infty)
\end{array}
\]
It follows that $f=L^{\half}(a)\f.$

Next it is shown that
\[
\lim\limits_{x\downarrow a}||L^{-\half}(x) L^{\half}(a)h||=||h||,\; h\in\cran L(a).
\]
Since $\ran L^{\half}(a) \subseteq \ran L^{\half}(x)$, one has
\[
\sup\limits_{g\in
H}\cfrac{|(L^{\half}(a)h,g)|^2}{(L(x)g,g)}=||L^{-\half}(x)L^{\half}(a)h||^2.
\]
Then for an arbitrary $\varepsilon>0$ there exists $g_\varepsilon$
such that
\[
||L^{-\half}(x)L^{\half}(a)h||^2-\varepsilon\le\cfrac{|(L^{\half}(a)h,g_\varepsilon)|^2}{(L(x)g_\varepsilon,g_\varepsilon)}
\le ||L^{-\half}(x)L^{\half}(a)h||^2.
\]
Let
\[
C:=\sup\limits_{x\in(a,b)}||L^{-\half}(x)L^{\half}(a)h||^2.
\]
Then by \eqref{domchar} $C<\infty$. Since
\[
\lim\limits_{x\downarrow
a}||L^{-\half}(x)L^{\half}(a)h||^2=\sup\limits_{x\in(a,b)}||L^{-\half}(x)L^{\half}(a)h||^2,
\]
we get
\[
C-\varepsilon\le\cfrac{|(L^{\half}(a)h,g_\varepsilon)|^2}{(L(a)g_\varepsilon,g_\varepsilon)}
\le C
\]
It follows that
\[
||h||^2=\sup\limits_{g\in
H}\cfrac{|(L^{\half}(a)h,g)|^2}{(L(a)g,g)}=C.
\]
Thus
\[
\lim\limits_{x\downarrow a}||L^{-\half}(x)L^{\half}(a)h||^2=||h||^2,\;
h\in\cran L(a).
\]
Relation \eqref{nomerod} is proved. Relation \eqref{nomerdva} can be
proved similarly.
\end{proof}

\begin{remark}\label{vspomRem}
If $L(x)$ is non-increasing on $(a,b)$ then $L^{-1}(x)$ is non-decreasing on $(a,b)$
and \eqref{nomerod} and \eqref{nomerdva} are replaced by
\[
 \lim\limits_{x\downarrow a} L^{-1}(x)[f]=L^{-1}(a)[f],\;f\in\cR_0 \, \subseteq \cD[L^{-1}(a)],
\]
\[
 \lim\limits_{x\uparrow b} L^{-1}(x)[f]=\left\{\begin{array}{l}L^{-1}(b)[f],\;f\in\cD[L^{-1}(b)],\\
+\infty,\qquad  f\in\cR_0\setminus\cD[L^{-1}(b)].
\end{array}\right.
\]
\end{remark}

\begin{proposition}
\label{jbfzo}
Let $\cQ(\lambda)\in\wt S(\sM)$ (respectively $\in \wt S^{(-1)}(\sM)$). Then the strong resolvent limits
\[
s-{\rm R}-\lim\limits_{x\uparrow 0}\cQ(x)=:\cQ(-0),\;s-{\rm R}-\lim\limits_{x\downarrow -\infty}\cQ(x)=:\cQ(-\infty)
\]
exist and here $\cQ(-0)$ and $\cQ(-\infty)$ are nonnegative (resp., nonpositive) selfadjoint l.r.'s in $\sM$.
In addition, for $\cQ(\lambda)\in\wt S(\sM)$ one has
\[
\begin{array}{l}
 \lim\limits_{x\uparrow 0}\cQ(x)[g]=\cQ(-0)[g], \; g\in\cD[\cQ(-0)], \\
 \lim\limits_{x\downarrow -\infty}\cQ(x)[g]=\cQ(-\infty)[g], \; g\in\cD[\cQ(-1)]\subseteq \cD[\cQ(-\infty)].
\end{array}
\]
and for $\cQ(\lambda)\in\wt S^{(-1)}(\sM)$ one has
\[
\begin{array}{l}
 \lim\limits_{x\uparrow 0}\cQ(x)[g]=\cQ(-0)[g], \; g\in\cD[\cQ(-1)]\subseteq \cD[\cQ(-0)], \\
 \lim\limits_{x\downarrow -\infty}\cQ(x)[g]=\cQ(-\infty)[g], \; g\in\cD[\cQ(-\infty)].
\end{array}
\]
\end{proposition}

\begin{proof}
Let $\cQ(\lambda)\in\wt S(\sM)$ and let $\Omega(z)\in\cRS(\sM)$ be as in Lemma \ref{TH1},
so that \eqref{formula1} holds.
Then the strong non-tangential limit values $\Omega(\pm 1)$ exist
and are selfadjoint contractions such that $\Omega(-1)\le \Omega(y)\le \Omega(1)$, $y\in (-1,1)$.
The equality \eqref{formula1} shows that
\[
 (\cQ(\lambda)+I)^{-1}=\cfrac{1}{2}\left(I-\Omega\left(\cfrac{1+\lambda}{1-\lambda}\right)\right)
\]
and here the values $L(y):=I-\Omega(y)$ are nonnegative and non-increasing for $y\in(-1,1)$.
By Theorem \ref{typB} $\cD[\cQ(x)]$, $x<0$, and hence also $\ran (I-\Omega(y))^{\half}$, $y\in(-1,1)$, is constant (cf. \eqref{ravenstva1}).
Hence it follows from Lemma \ref{vspom} and Remark \ref{vspomRem}, cf. also Proposition \ref{clfrm},
that
\[
 \cQ(-0)=s-{\rm R}-\lim\limits_{x\uparrow 0}\cQ(x)=-I+2(I-\Omega(1))^{-1}
\]
and $\lim\limits_{x\uparrow 0}\cQ(x)[g]=\cQ(-0)[g]$, $g\in\cD[\cQ(-0)]$. Similarly one has
\[
 \cQ(-\infty)=s-{\rm R}-\lim\limits_{x\downarrow -\infty}\cQ(x) = -I+2(I-\Omega(-1))^{-1}
\]
and $\lim\limits_{x\downarrow -\infty}\cQ(x)[g]=\cQ(-\infty)[g]$, $g\in\cD[\cQ(-1)]\subseteq \cD[\cQ(-\infty)]$.

In the case that $\cQ(\lambda)\in\wt S^{(-1)}(\sM)$ one applies formula \eqref{formula2} in Lemma~\ref{TH1},
\[
 (\cQ(\lambda)-I)^{-1}=-\cfrac{1}{2}\left(I+\Omega\left(\cfrac{1+\lambda}{1-\lambda}\right)\right).
\]
Here the values $L(y):=I+\Omega(y)$ are nonnegative, non-decreasing for $y\in(-1,1)$,
and by Theorem \ref{typB} the domain $\cD[\cQ(x)]$, $x<0$, thus also $\ran (I+\Omega(y))^{\half}$, $y\in(-1,1)$, is constant.
Now the statements follow from Lemma \ref{vspom} and Proposition \ref{clfrm}.
\end{proof}

It is possible that $\Omega(1)=I$ and therefore e.g. for a Stieltjes family one can have $\cQ(-0)=\{0\}\times\sM$;
also the inclusion $\cD[\cQ(-1)]\subseteq \cD[\cQ(-\infty)]$ can be strict.
Consider, for instance, a Stieltjes function $\cQ(\lambda)=-\cfrac{1}{\lambda}\, H$, $\lambda\neq 0$,
where $H\geq 0$ is an unbounded selfadjoint operator in the Hilbert space $\sM$ with $\ker H=\{0\}$.
Then $\cQ(-0)=\{0\}\times\sM$ with $\cD[Q(-0)]=\{0\}$ and $\cQ(-\infty)=0$, the zero operator on $\sM$,
so that $\cD[\cQ(\lambda)]=\dom H^{\half} \subset \cD[\cQ(-\infty)]=\sM$.


\end{document}